\documentclass[11pt,reqno]{amsart}

\usepackage{color,graphicx,mathrsfs,tikz,amssymb}
\usetikzlibrary{calc}
\usepackage{hyperref}\hypersetup{
    colorlinks=true,
    linkcolor={blue!50!black},
    filecolor=red,      
    urlcolor=red,
    citebordercolor=red,
    citecolor={blue!50!black}
}

\numberwithin{equation}{section}

\usepackage{amsthm}

\newcommand{\rev}[1]{{{#1}}}

\renewcommand{\Omega}{\om}
\newcommand{\veps}{\varepsilon}
\newcommand{\divx}{\mathrm{div}_x}
\newcommand{\vPi}{\varPi}
\newcommand{\vPih}{\hat{\varPi}}
\newcommand{\vSig}{\varSigma}
\newcommand{\hK}{\hat K}
\newcommand{\ran}{\mathop{\mathrm{ran}}}
\newcommand{\Vsp}{\rprp{(V^*)}}

\newcommand{\rprp}[1]{{#1}^\perp}
\newcommand{\lprp}[1]{\sideset{^\perp}{}{\mathop{#1}}}
\newcommand{\ip}[1]{\langle {#1} \rangle}
\newcommand{\Gt}{\tilde{G}}
\newcommand{\vt}{\tilde{v}}
\newcommand{\ut}{\tilde{u}}
\newcommand{\wt}{\tilde{w}}
\newcommand{\vG}{\varGamma}
\newcommand{\om}{\varOmega}

\newcommand{\oh}{\varOmega_h}
\newcommand{\eh}{\mathcal{E}_h}
\newcommand{\og}{\omega}
\newcommand{\D}{\mathcal{D}}

\newcommand{\V}{\mathcal{V}}

\newcommand{\CCC}{\mathbb{C}}
\def\d{\partial}
\newcommand{\dom}{\mathop{\mathrm{dom}}}

\newcommand{\dt}{\partial_t }

\newcommand{\dxx}{\partial_{xx} }

\newtheorem{theorem}{Theorem}[section]
\newtheorem{lemma}[theorem]{Lemma}

\theoremstyle{definition}

\theoremstyle{remark}
\newtheorem{assumption}{Assumption}
\newtheorem{remark}[theorem]{Remark}
\newtheorem{problem}[theorem]{Problem}

\numberwithin{equation}{section}

\newcommand{\diam}{\mathop{\mathrm{diam}}}

\newcommand{\RRR}{\mathbb{R}}

\begin{document}

\title{A spacetime DPG method for the Schr\"odinger equation}

\author[L. Demkowicz]{L.~Demkowicz}
\address{Institute for Computational Engineering and Sciences,
  The University of Texas at Austin, Austin, TX 78712, USA}
\email{leszek@ices.utexas.edu}

\author{J.~Gopalakrishnan}
\address{Portland State University,  PO Box 751, Portland, OR 97207-0751}\email{gjay@pdx.edu}

\author{S.~Nagaraj}
\address{Institute for Computational Engineering and Sciences,
  The University of Texas at Austin, Austin, TX 78712, USA}
\email{sriram@ices.utexas.edu}

\author{P. Sep\'{u}lveda}
\address{Portland State University,  PO Box 751, Portland, OR 97207-0751}\email{spaulina@pdx.edu}

\thanks{Corresponding author: Sriram Nagaraj (sriram@ices.utexas.edu)}
\thanks{This work was partly supported by \rev{AFOSR (FA9550-17-1-0090)}, NSF (DMS-1418822 and DMS-1318916) \rev{and} ONR (N00014-15-1-2496)}

\begin{abstract}
A spacetime discontinuous Petrov-Galerkin (DPG) method for the linear time-dependent Schr\"odinger equation is proposed. The spacetime approach is particularly attractive for capturing irregular solutions. Motivated by the fact that some irregular Schr\"odinger solutions cannot be solutions of certain first order reformulations, the proposed spacetime method uses the second order Schr\"odinger operator. Two variational formulations are proved to be well posed: a strong formulation (with no relaxation of the original equation) and a weak formulation (also called the ``ultraweak formulation'', which transfers all derivatives onto test functions). The convergence of the DPG method based on the ultraweak formulation is investigated using an interpolation operator. A stand-alone appendix analyzes the ultraweak formulation for general differential operators. Reports of numerical experiments motivated by pulse propagation in dispersive optical fibers are also included.
\end{abstract}

\maketitle

\section{Introduction}

This paper is devoted to exploring a weak formulation and an
accompanying numerical technique for the Schr\"odinger equation with
Dirichlet boundary conditions.  Let $\Omega_0 \subset \mathbb{R}^n$
$(n\ge 1)$ be an open bounded domain with Lipschitz boundary.  The
space variable $x$ lies in $\om_0$ while the time variable $t$ lies in
the open interval $(0,T)$ with $T < \infty$.  The classical form of
Schr\"odinger initial boundary value problem reads as follows:
\begin{subequations}
  \label{eq:ls2}
    \begin{align}
      \label{eq:ls2pde}
    i \d_t u - \Delta_x u &= f,
     && \quad x \in \om_0,\;  0< t < T,
     \\
     u(x,t) &= 0,
     && \quad x \in \d\om_0, \; 0< t < T,
     \\
     u(x,0) &= 0,      && \quad x \in \om_0,
    \end{align}
\end{subequations}
where $\d_t$ denotes the time derivative $\d/\d t$ and 
$\Delta_x$ denotes the Laplacian with respect to the spatial
variable~$x$.  Here $f$ is any given function in $ L^2(\Omega)$ and
$\om = \om_0 \times (0,T)$ throughout.

The numerical technique we want to apply to~\eqref{eq:ls2} is the
discontinuous Petrov-Galerkin (DPG) method~\cite{DemkoGopal11}.  Among
its desirable properties are mesh-independent stability, inheritance
of discrete stability from the well-posedness of the undiscretized
problem, and the availability of a canonical error indicator computed
as part of the solution. The DPG method has been successfully applied
to a wide variety of problems such as second order elliptic
problems~\cite{dpg1}, convective
phenomena~\cite{dahmen,DemkoGopal10a,DemkoHeuer13},
elasticity~\cite{BramwDemkoGopal12,CarstDemkoGopal14,jay1,elas},
Stokes flow~\cite{CarstDemkoGopal14,stokes14}, and spacetime
problems~\cite{spacetime2, spacetime1, Wiene16}.  It seems natural
therefore that the DPG method should work for~\eqref{eq:ls2} as well.
In this paper, we will show that the DPG method does indeed
faithfully approximate the solutions of~\eqref{eq:ls2} {\em provided}
we do not recast~\eqref{eq:ls2} into a first order system.

Many applications of interest come  a first order systems even if
they are often displayed as second order partial differential
equations. For example, the second order heat equation is really a
combination of two first order equations, namely the Fourier law of
heat conduction and the conservation of energy. Similarly, the linear
elasticity equation, while often displayed as a second order equation
for displacement, is really a combination of two first order
equations, the constitutive (Hooke's) law and the equation of static
equilibrium.  Thus it's no surprise that it makes physical sense to
return the heat equation or the elasticity equation to first order
form before discretizing.  However, it makes no physical sense to do
this for the Schr\"odinger equation, which is not derived from first
order physical laws.

It makes no mathematical sense either. One might be tempted to
introduce a ``flux'' $\tau,$ formulate the first order system
$ i \d_t u - \divx \tau = f$ and $\nabla_x u - \tau = g,$ and claim
the latter's equivalence to~\eqref{eq:ls2} when $g=0.$ This claim is false
because, while the Schr\"odinger problem~\eqref{eq:ls2} is well-posed
for $f \in L^2(\om)$, {\em the first order system cannot be well-posed
  in $L^2(\om).$} Indeed, denoting the norm of $L^2(\om)$ by
$\| \cdot \|_\om$, if the first order system were well-posed, then
there would be constants $C_1,C_2>0$ such that
$ \| u\|_\om + \| \tau \|_\om \leq C_1\|f\|_\om + C_2\|g\|_\om.$ But
then the second equation of the system implies that
$\| \nabla_x u \|_\om = \| g + \tau \|_\om\le C_1\|f\|_\om +
2C_2\|g\|_\om$ for any solution~$u$, which is false: In the next two paragraphs
we will exhibit a Schr\"odinger
solution for which $\| \nabla_x u \|_\om = \infty$ even when $g=0$ and
$f \in L^2(\om)$.

First observe that given any $f(x,t)$ in $L^2(\om)$, it is possible to
solve~\eqref{eq:ls2} by the ``method of Galerkin
approximations''~\cite{Evans98} (distinct from the Galerkin finite
element method).  Let $e_k(x)$ in $H_0^1(\om_0)$ and $\og_k^2>0$ be an
eigenpair of $\Delta_x$ satisfying
\begin{equation} \label{eigenproblem}
-\Delta_x e_k
=\omega_k^2 e_k\qquad   \text {  a.e. in } \om_0,
\end{equation} 
normalized so that $\|e_k\|_{\om_0} = 1$ for all
natural numbers $k\ge 1$.
Since Fubini's theorem for product measures implies that $f(\cdot, t)$
is in $L^2(\om_0)$, the following definitions make sense:
\begin{subequations}
  \label{eq:fkukFMUM}
\begin{align}
  f_k(t) & 
           = \int_{\om_0} f(x,t) \bar{e}_k(x) \; dx,
  &
    u_k(t) & = -i \int_0^t e^{i\og_k^2 (t-s) } f_k(s)\; ds,
  \\
  F_M(x,t)  &= \sum_{k=1}^{M}f_k(t) e_k(x),
  &
    U_{M}(x,t)  &= \sum_{k=1}^{M}u_k(t) e_k(x).
\end{align}
\end{subequations}
It is not difficult to show (see the proof of Theorem~\ref{thm:strong}
below) that $u = \lim_{M \to \infty} U_M$
exists in $L^2(\om)$ and solves~\eqref{eq:ls2}. 

Now consider the one-dimensional case $\om_0 = (0,1)$, where $\og_k = k
\pi$, and choose
\[
f(x,t) = \sum_{k=1}^\infty 
\frac{1}{k} e^{i \omega_k^2 t} e_k(x) \quad\text{ in } L^2(\om).
\]
Then by the orthonormality of $e_k$, we have that
$f_k(s) = e^{i \omega_k^2 s}/k$, 
$u_k(t) = -i t e^{i \og_k^2 t}/k$,
\begin{align*}
\| U_M \|_{\om}^{2} 
  & =  \sum_{k=1} ^M \int_0 ^T |u_k(t)|^2\, dt
  = \sum_{k = 1} ^M \int_0 ^T\left| \frac{ -i t}{k}e^{i\omega_k ^2
    t}\right|^2 dt
    =\frac{T^3}{3} \sum_{k=1}^M \frac{1}{k^2},
  \\
  \| \nabla_x U_M\|^2_\om
  & = \sum_{k = 1}^M \og_k^2\int_0 ^T \left| u_k(t)\right|^2\, dt 
    =   T^3\sum_{k=1}^M  \frac{\omega_k ^2}{3k^2} 
    =   \frac{\pi^2 }{3} T^3 M.  \label{imposible}
\end{align*}
The solution $u$ is the limit of $U_M$. The above calculations clearly
show that as $M\to \infty$, while
$\| u\|_\om = \lim_{M\to \infty} \| U_M \|_{\om} = (T^3\pi^2/18)^{\frac{1}{2}}$, the
limit of $\| \nabla_x U_M\|_\om$ diverges. {\em Thus it is possible to
  obtain a Schr\"odinger solution $u$ whose $H^1$-norm is infinite
  even when $f \in L^2(\om)$.}  Note that finer arguments are needed
to understand the regularity of Schr\"odinger solutions in unbounded
domains, which although a topic of wide mathematical
interest~\cite{Tao06}, is not our concern here.

To our knowledge, this paper is the first work to analyze the
feasibility of the DPG methodology for a system \rev{without ready access
to an equivalent  
first order formulation.} The second order form necessitates
formulations in the nonstandard graph spaces of the second order
Schr\"odinger operator.  One of the contributions of this paper is the
proof of well-posedness of a strong and a weak formulation
of~\eqref{eq:ls2} in these graph spaces.  The general spaces and
arguments required for this analysis are collected in a stand-alone appendix
(Appendix~\ref{apn:weak}), anticipating uses outside of the Schr\"odinger
example. The analysis in Appendix~\ref{apn:weak} is motivated by the
modern theory of Friedrichs systems~\cite{guermond1} but applies
beyond Friedrichs systems.  Borrowing the approach
of~\cite{guermond1}, we are able to prove well-posedness without
developing a trace theory for the graph spaces.  The other
contributions involve the numerical implications of this
well-posedness.  Numerical methods using the strong formulation must
use conforming finite element subspaces of the graph spaces. On the
other hand, numerical methods using the weak formulation need only use
existing standard finite element spaces. In either case, an
interpolation theory in the Schr\"odinger graph norm is needed to
estimate convergence rates. We address this issue in one space
dimension.

In the next section, we investigate well-posedness (in the sense of
Hadamard) for a strong and weak variational formulation for the
Schr\"odinger problem.  This will require an abstract definition of a
boundary operator and duality pairings in a graph space. Such abstract
definitions that apply beyond the Schr\"odinger setting are in
Appendix~\ref{apn:weak}. Their particular realizations for the
Schr\"odinger case are used in section~\ref{sec:Schrodinger}. (To avoid
repetitions of the general definitions in the specific case, we will
often refer to Appendix~\ref{apn:weak} in
section~\ref{sec:Schrodinger}.)  Section~\ref{sec:dense} provides a verification of a
density assumption made in section~\ref{sec:Schrodinger}.  Section~\ref{sec:error} details our
construction of a conforming finite element space and interpolation
error estimates. Section~\ref{sec:results} points to an application in
dispersive optical fibers and contains some numerical results.

\section{Functional Setting and Wellposedness} \label{sec:Schrodinger}

We now provide a functional setting within which a strong and a weak
formulation of the spacetime Schr\"odinger problem can be proved to be
well posed (i.e., \emph{inf-sup} stable). The analysis is an application
of the general theory detailed in Appendix~\ref{apn:weak}.

The classical form of the problem is already presented
in~\eqref{eq:ls2}.  Recalling that $\Omega=\Omega_0 \times (0,T)$,
define these parts of $\d\om$:
\[
\vG = \partial \Omega_0 \times [0,T] \cup \Omega_0 \times
\{0\}, \qquad 
\vG^* = \partial \Omega_0 \times [0,T] \cup \Omega_0 \times
\{T\}
\]
(see Figure~\ref{fig:domain}).  Then the initial and boundary conditions
together can be written as $u|_\vG =0$.  We want to
write~\eqref{eq:ls2} as an operator equation (see~\eqref{eq:bvp}) to
apply the general results of Appendix~\ref{apn:weak}. To this end,
consider the setting of Appendix~\ref{apn:weak} with
\begin{gather*}
  A =  A^* =  i \partial_t - \Delta_x ,
  \qquad k=l=m=1, \; d=n+1.
\end{gather*}
The space $W = W^*$ is then defined by~\eqref{eq:W}--\eqref{eq:W*}; \rev{namely,
$W = W^* = \{ u \in L^2(\om): \; i\d_t u - \Delta_x u \in L^2(\om) \}$.
The operator $D=D^*: W \to W'$} is defined
by~\eqref{eq:D}--\eqref{eq:D*}; namely,
\rev{
$  \ip{ Dw, \wt}_{W} = (Aw, \wt)_\om - (w, A\wt)_\om$ for all $w, \wt \in W.$}
As usual, let $\D(\bar \om)$ denote
the restrictions of functions from $\D(\RRR^{n+1})$ to $\om$.  The operator $D$
(often called the ``boundary operator'' in the theory of Friedrichs
systems~\cite{guermond1}), satisfies
\begin{equation}
  \label{eq:Dsmoooth}
  \ip{D \phi, \psi}_W
  =\int_{\d \om} i n_t \phi \bar \psi
    + \int_{\d\om} \phi (n_x \cdot \nabla_x \bar \psi)
    - \int_{\d\om} (n_x \cdot \nabla_x \phi) \bar \psi
\end{equation}
for all $\phi, \psi \in \D(\bar \om)$.  Note that \rev{although} 
the integrals on the
right-hand side \rev{need not} exist for all functions in $W$, \rev{$D$ 
is} defined on all $W$ through~\eqref{eq:D}.

Although we set the differential operators $A$ and $A^*$ to be equal
above, note that we consider each as an unbounded operator with its
own domain.  We set the domain of $A$ to
\begin{equation} \label{eq:domdef}
\dom(A) =
\{u \in W : \langle Dv, u \rangle_{W} = 0, \, 
\forall v \in \mathcal{V}^*\}, 
\end{equation}
where
$
\V^* = \{ \varphi \in \D(\bar\om): \; 
\varphi|_{\vG^*} =0
\}.
$
The domain of the adjoint is given by the
usual~\cite{Brezi11,OdenDemko10} general prescription:
$\dom(A^*) = \{ s \in L^2(\om)^l:$  $\exists\,\ell \in L^2(\om)^m $
such that $(Av, s)_\om = (v, \ell)_\om$  $\forall v \in \dom(A)\}.$
Here $(\cdot,\cdot)_\om$ denotes the (complex) inner product in
$L^2(\om)$.  Finally, as in the appendix, set $V=\dom(A)$ and
$V^*=\dom(A^*)$ with the understanding that both $V$ and $V^*$ are
endowed with the $W$-topology, while $\dom(A)$ and $\dom(A^*)$ have
the topology of $L^2(\om)$.

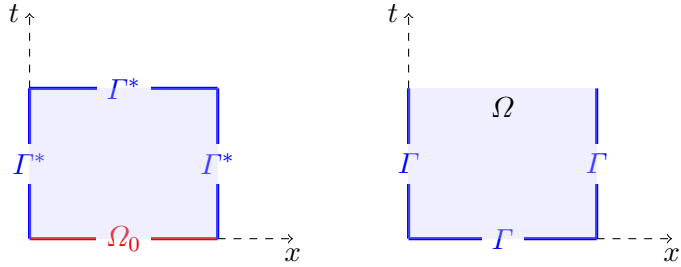
\begin{figure}
\begin{center}

\begin{tikzpicture}
\coordinate (top) at  (0,2);
\coordinate (bot) at  (2.5,0);
\coordinate (lft) at  (0,0);
\coordinate (rgt) at  (2.5,2);

\draw [red!90!black, very thick]
  (lft)  -- node[midway, fill=white] {$\om_0$} (bot) ;
\draw [blue, very thick] 
  (top)  -- node[midway, fill=white] {$\vG^*$} (rgt) ;
\draw [blue, very thick] 
  (rgt)  -- node[midway, fill=white] {$\vG^*$} (bot) ;
\draw [blue, very thick] 
  (top)  -- node[midway, fill=white] {$\vG^*$} (lft) ;

\fill [fill=blue!20, fill opacity=0.3] 
(bot) -- (rgt) -- (top) --(lft) ;
\draw[->,dashed] (2.5,0) -- ($(3.5,0)$) node[below] {$x$};
\draw[->,dashed] (0,2) -- ($(0,3)$) node[left] {$t$};
\end{tikzpicture}
\qquad
\begin{tikzpicture}
\coordinate (top) at  (0,2);
\coordinate (bot) at  (2.5,0);
\coordinate (lft) at  (0,0);
\coordinate (rgt) at  (2.5,2);


\draw [blue, very thick] 
  (rgt)  -- node[midway, fill=white] {$\vG$} (bot) ;
\draw [blue, very thick] 
  (top)  -- node[midway, fill=white] {$\vG$} (lft) ;
\draw [blue, very thick] 
  (lft)  -- node[midway, fill=white] {$\vG$} (bot) ;

\fill [fill=blue!20, fill opacity=0.3] 
(bot) -- (rgt) -- (top) --(lft) ;

\node at (1.25,1.75) {$\om$};
\draw[->,dashed] (2.5,0) -- ($(3.5,0)$) node[below] {$x$};
\draw[->,dashed] (0,2) -- ($(0,3)$) node[left] {$t$};
\end{tikzpicture}
 
\caption{Schematic of the spacetime domain}
\label{fig:domain}
\end{center}
\end{figure}

\rev{For the above set $A, A^*$ and $\dom(A)$, 
the conditions~\eqref{eq:A1} and \eqref{eq:A3}
in Appendix~\ref{apn:weak} are immediate, while 
condition~\eqref{eq:A2} is easily verified using~\eqref{eq:Dsmoooth}.}
Hence Lemma~\ref{lem:V1} shows
that $\dom(A^*)$ equals
\begin{equation}
  \label{eq:V*DV}
  V^* = \lprp{ D(V) }.  
\end{equation}
Rewriting~\eqref{eq:domdef} in the same style, 
\begin{equation}
  \label{eq:domA}
  V = \lprp{ D(\V^*)}.  
\end{equation}
Thus $V$ and $V^*$ are closed subspaces of $W$.  Let
$\V = \{ \varphi \in \D(\bar\om): \; \varphi|_{\vG} = 0\}$.

\begin{lemma}   \label{lem:VVsubset}
  $\V \subset V$ and $\V^* \subset V^*$.
\end{lemma}
\begin{proof}
  Equation~\eqref{eq:Dsmoooth} implies $\ip{D \phi,\phi^*}_W =0$ for
  all $\phi \in V$ and $\phi^* \in \V^*.$ Hence, any $\phi \in \V$ is
  also in $\lprp{D(\V^*)},$ which by~\eqref{eq:domA} implies that
  $\phi \in V$. Thus $\V \subset V$.

  If $\phi^* \in \V^*$, then~\eqref{eq:domdef} shows that
  $\ip{ D \phi^*, u}_W=0$ for all $u \in V$, i.e.,
  $\phi^* \in \lprp{D(V)},$ which by~\eqref{eq:V*DV} implies $\phi^*$
  is in $V^*$.
\end{proof}

In~\eqref{eq:domA}, $\V^*$ may be replaced by $V^*,$ provided that a density
result is available, as we show next. 

\begin{assumption} \label{asm:dense}
  Suppose that $\V^*$ is dense in $V^*$ and that $\V$ is dense in $V$.
\end{assumption}
\begin{lemma} \label{lem:VDV*}
  If Assumption~\ref{asm:dense} holds, then 
  $
  \lprp{D(\V^*)} = \lprp{D(V^*)}
  $
  and 
  $\lprp{D(\V)} = \lprp{D(V)}.$
\end{lemma}
\begin{proof}
  Clearly $\V^* \subseteq V^*$ implies
  $ \lprp{D(\V^*)} \supseteq \lprp{D(V^*)}.$ To prove the reverse
  containment, suppose $ w \in \lprp{D(\V^*)}$ and $v^* \in V^*$. By
  density, there is a sequence $\vt_n$ in $\V^*$ satisfying
  $\lim_{n\to \infty} \| \vt_n - v^*\|_W =0$. Since
  $\ip{ D \vt_n, w}_W=0,$ by the continuity of $D$, we have
  $\ip{ D v^*, w}_W=0.$ Hence $w \in \lprp{D (V^*)}$. The second
  identity is proved similarly.
\end{proof}

\subsection{Strong formulation}

The strong formulation of the Schr\"odinger problem~\eqref{eq:ls2} is
based on these sesquilinear and conjugate linear forms:
\[
a(u,v) =(Au,v)_\Omega, \qquad 
l(v) = (f,v)_\Omega.
\]
\begin{problem}[Strong formulation] \label{prb:strong}
  {\em Given any $f \in L^2(\om)$, find 
  $u \in V$ satisfying 
  \[
  a(u,v) = l(v), \qquad\, \forall v \in L^2(\om).
  \]}
\end{problem}

\begin{theorem}
  \label{thm:strong}
  Suppose that Assumption~\ref{asm:dense} holds. Then the linear Schr{\"o}dinger
  operator $A: V \rightarrow L^2(\Omega)$ is a continuous
  bijection. Hence Problem~\ref{prb:strong} is well-posed.
\end{theorem}
\begin{proof}
  To prove the surjectivity of $A$, suppose $f \in L^2(\om)$. Recall
  the definitions of $e_k$, $u_k$ and $f_k$ from~\eqref{eigenproblem}
  and~\eqref{eq:fkukFMUM}.
  Clearly, $AU_M = F_M$. Since $U_M$ and any $\varphi \in \V^*$ are
  smooth enough for integration by parts using $\varphi|_{\vG^*}=0$
  and $U_M|_\vG=0,$ we have 
  \begin{align*}
    (i \d_t U_M, \varphi)_\om 
    & = (U_M, i \d_t \varphi)_\om, \\
    (\Delta U_M, \varphi)_\om
    & = (U_M, \Delta \varphi)_\om.
  \end{align*}
  Hence
  $\ip{D \varphi, U_M}_W = (A \varphi, U_M)_\om - (\varphi, A
  U_M)_\om=0$
  for all $\varphi \in \V^*.$ By~\eqref{eq:domA}, this implies that
  $U_M$ is in $V$.

  Next, we show that $U_M$ is a Cauchy sequence in $V$. For any $N>M$,
  \begin{align*}
  \| U_M - U_N \|_\om^2 
    & = 
      \sum_{k=M+1}^N \int_0^T |u_k(t)|^2 \; dt
      \le 
      \frac{1}{2}T^2 \sum_{k=M+1}^\infty   \int_0^T |f_k(t)|^2 \; dt,
    \\
    \| A(U_M - U_N) \|_\om^2 
    & = \| F_M - F_N \|_\om^2 
      \le \sum_{k=M+1}^\infty \int_0^T | f_k(t)|^2 \; dt,
  \end{align*}
  both of which converge to 0 as $M\to \infty$, because
  $f \in L^2(\om)$. Thus $U_M$ is Cauchy. It must therefore have an
  accumulation point $u$ in $V$. Moreover, since $Au$ and $f$ are
  $L^2(\om)$-limits of the same sequence $F_M = AU_M$, we have
  $A u = f$. Thus $A: V \to L^2(\om) $ is surjective.

  We use a similar argument (with $u_k$ defined by integrals from $T$
  to $t$) to show that $A=A^* : V^* \to L^2(\om)$ is also
  surjective. We omit the details, but note that 
  the only difference is that instead of~\eqref{eq:domA},
  we must now use
  \[
  V^* = \lprp{ D (\V)}, 
  \]
  which follows from~\eqref{eq:V*DV}, Assumption~\ref{asm:dense} and
  Lemma~\ref{lem:VDV*}. Finally, since $\ker(A) = \rev{\lprp{\ran(A^*)}}$,
  the surjectivity of $A^* : V^* \to L^2(\om)$ shows that
  $A: V \to L^2(\om)$ is injective, thus completing the proof of the
  stated bijectivity.
\end{proof}

\begin{remark}
  An example of a standard well-posedness result for the Schr\"odinger
  equation obtained using semigroup theory is
  \cite[Theorem~4.8.1]{Kesav89}, which proves that there is one and
  only one solution to~\eqref{eq:ls2} whenever $f\equiv 0$ and
  $u(x,0)$ is in $H^2(\om) \cap H_0^1(\om)$.  In
  Theorem~\ref{thm:strong}, we have shown (by a different method) that
  the existence of a unique solution holds for any $f \in L^2(\om)$.
  Note that in the above proof, we used Assumption~\ref{asm:dense}
  only to obtain injectivity.  If one opts to use the results
  of~\cite{Kesav89} (with $u(x,0)\equiv 0$ and $f\equiv 0$) to
  conclude injectivity, then there is no need to place
  Assumption~\ref{asm:dense} in Theorem~\ref{thm:strong}.
\end{remark}

\subsection{A weak formulation}

Now we consider a mesh-dependent weak formulation that is the basis
of the DPG method. This formulation, sometimes called the
``ultraweak'' formulation, is given in a general setting in
Problem~\ref{prb:weak} of Appendix~\ref{apn:weak}.  We apply this to 
 our example of
the Schr\"odinger equation. 

The spacetime domain $\om$ is partitioned into a mesh $\oh$ of
finitely many open elements $K$ such that
$\bar \om = \cup_{K \in \oh} \bar K$ where
$h=\max_{K \in \oh} \diam(K)$.  \rev{Particularizing the general
  definitions in Appendix~\ref{apn:weak} (see~\eqref{eq:WhWh*}
  through~\eqref{eq:Q}) to the Schr\"odinger example, we let
  $A_h=A_h^*$ be the Schr\"odinger operator applied element by element
  and let $W_h = W_h^* = \{ w \in L^2(\om): A (w|_K) \in L^2(K)$ for
  all $K \in \oh\}$.  The operator $D_h : W_h \to W_h'$ is defined by
  $\ip{D_h w, v}_{W_h} = (A_h w,v)_\om - (w, A_h v)_\om$ for all
  $w,v \in W_h$ and let $D_{h,V}: V \to W_h'$ be denote $D_h|_V$. The
  range of $D_{h,V}$, denoted by $Q,$ is made into a complete space by
  the norm
  $\| q\|_Q = \inf_{ v \in D_{h,V}^{-1}(\{ q \}) } \| v \|_W$.
  Abbreviating the duality pairing $\ip{\cdot,\cdot}_{W_h}$} by
  $\ip{ \cdot, \cdot}_h$, define the sesquilinear form
  $ b( (u,q), v) = (u, A_h v)_\om + \ip{ q,v }_h$ on
  $(L^2(\om) \times Q) \times W_h$.

\begin{problem}[Ultraweak formulation] \label{prb:weakSchrodinger}{\em 
  Given $F \in W_h',$\,  find $u\in L^2(\om)$
  and $q \in Q$ such that
  \[
    b((u,q), v)= F(v)
  \qquad
  \forall v \in W_h.    
  \]}
\end{problem}

\begin{theorem} \label{thm:uw_schrod}
  Suppose that Assumption~\ref{asm:dense} holds. Then
  Problem~\ref{prb:weakSchrodinger} is well-posed, i.e., there is a
  $C>0$ such that given any $F\in W_h'$, there is a unique solution
  $(u,q) \in L^2(\om) \times Q$ to Problem~\ref{prb:weakSchrodinger}
 that satisfies
  \[
  \| u \|_\om^2 + \|q \|_Q^2 \;\le\, C\, \| F \|_{W_h'}^2.
  \]
  Moreover, if $F(v) = (f,v)_\om$ for some $f \in L^2(\om),$ then~$u$
  is in~$V$ and $q= D_hu$.
\end{theorem}
\begin{proof}
  The result follows from   Theorem~\ref{thm:wellposed-mesh}.
  Since Lemma~\ref{lem:VDV*} together with~\eqref{eq:domA}
  implies~\eqref{eq:asm:1} and since Theorem~\ref{thm:strong}
  implies~\eqref{eq:asm:2}, the assumptions of
  Theorem~\ref{thm:wellposed-mesh} are verified. 
\end{proof}

\section{Verification of the density assumption}  \label{sec:dense}

In the next three sections, $\om_0$ is set to be the  interval $(0,L)$
and $\om = (0,L) \times (0,T)$ where $L, T>0.$ The purpose of this
section is to verify the density assumption
(Assumption~\ref{asm:dense}) in this case of one space dimension.

\begin{theorem}
  \label{thm:dense}
  Let $\om =(0,L)\times(0,T)$. Then
  $\V^*= \{ \varphi \in \D(\bar \om): \varphi|_{\vG^*} =0\}$ is dense
  in $V^*$, and $\V= \{ \varphi \in \D(\bar \om): \varphi|_{\vG} =0\}$
  is dense in $V.$
\end{theorem}
\begin{proof}
  Since the proofs of both the stated density results are similar, we
  will only show the proof of density of $\V^*$ in $V^*$. 

  {\em {\underline{\smash[b]{Step~1. Extend:}}}} Let $\om_l= (-L,0]\times (0,T)$
  and $\om_{r} =[L,2L)\times (0,T)$.  Define an operator $G$ that
  extends functions on $\om$ to
  $\hat\om \equiv \om_l \cup \om \cup \om_r$ by
  \[ 
  G w (x,t) =
  \left\{ 
    \begin{aligned}
      -& w(-x,t),    &&\; (x,t) \in \om_l,
      \\
      -& w(2L-x,t),  &&\; (x,t) \in \om_r,
    \end{aligned}
  \right.
  \]
  (and $Gw (x,t)= w(x,t)$ for all $(x,t) \in \om$).  Let $G'$ be the
  reverse operator that maps functions on $\hat\om$ to $\om$ by
  $ G' \bar w (x,t) = \bar w(x,t) - \bar w (-x,t) - \bar w(2L - x, t)
  $
  for all $(x,t) \in \om$ (see Figure~\ref{fig:density}). Such
  definitions are to be interpreted a.e., so that, for example, $Gw$ is
  well defined for any $w$ in $L^2(\om)$.  It is easy to see by a
  change of variable that
  \begin{equation}
    \label{eq:GG'}
    (Gf, g)_{\hat\om} = (f, G'g)_\om,\qquad
    \forall f \in L^2(\om), \; g \in L^2(\hat\om).  
  \end{equation}

  Next, we claim that
  \begin{equation}
    \label{eq:commuteAG}
    A G v = G A v, \qquad \forall v \in V^*.
  \end{equation}
  Clearly, $Gv$ is in $L^2(\hat\om)$.  Let $\varphi \in \D(\hat\om)$.
  \rev{Let $\ip{ AG v, \varphi}_{\D(\hat\om)}$ denote the
  action of the distribution $AG v$ on $\bar \varphi$. Then}
  $
    \ip{ AG v, \varphi}_{\D(\hat\om)}
     =  (Gv, A \varphi)_{\hat\om}
      = (v,   G'A \varphi)_\om$
  because of~\eqref{eq:GG'}. By the chain rule applied to the smooth
  function $\varphi$, we find that 
  \begin{equation}
    \label{eq:G'A}
      G' A \varphi = A G' \varphi.
  \end{equation}
  Hence,
  \begin{align}
    \label{eq:AG}
  \ip{ AG v, \varphi}_{\D(\hat\om)}
    &
      = ( v, A G'\varphi)_\om
      = (Av, G'\varphi)_\om - \ip{ D v, G'\varphi}_W.
  \end{align}
  Now observe that $G'\varphi|_\vG = 0$. Hence, by
  Lemma~\ref{lem:VVsubset}, $G'\varphi$ is in $V$.  Since $v \in V^*$
  and $G'\varphi \in V$, the last term of~\eqref{eq:AG} must vanish
  by~\eqref{eq:V*DV}. Thus
  $\ip{ AG v, \varphi}_{\D(\hat\om)} = (Av, G'\varphi)_\om = (G A v,
  \varphi)_{\hat \om},$
  completing the proof of the claim~\eqref{eq:commuteAG}.  In view
  of~\eqref{eq:commuteAG}, we conclude that $Gv$ is in $W(\hat \om)$
  whenever $v \in V\rev{^*}$.

  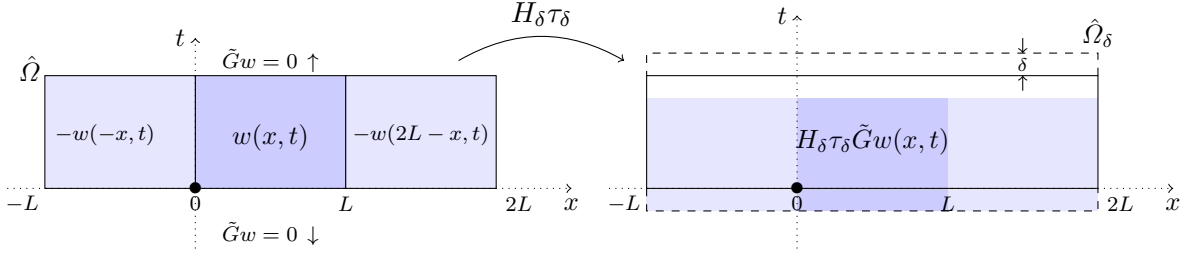
\begin{figure}
    \begin{center}

      \begin{tikzpicture}
        \draw [fill=blue!20] (2,0)--(2,1.5)--(4,1.5)--(4,0)--cycle; 
        \draw [fill=blue!10] (0,0) --(0 ,1.5)--(2,1.5)--(2,0)--cycle;
        \draw [fill=blue!10] (4,0) --(4,1.5)--(6,1.5)--(6,0)--cycle;
        \draw [ thick] (0.8,  1) node [ below]{ \fontsize{8}{12}\selectfont $-w(-x,t)$};
        \draw [ thick] (5, 1) node [ below]{ \fontsize{8}{10}\selectfont $-w(2L-x,t)$};
        \draw [ thick] (3,  1) node [ below]{ \fontsize{10}{15}\selectfont $w(x,t)$};
        \draw [black] (-.2,  1.9) node [ below]{ \fontsize{10}{14}\selectfont $\hat\om$};
        \draw [->,dotted] (2, -0.5-.3) -- (2, 2) node[left]{  \fontsize{10}{15}\selectfont$t$};
        \draw [->,dotted] (-0.5, 0) -- (7, 0) node[below]{  \fontsize{10}{15}\selectfont$x$};
        %
        \node[] at (-.3,-0.2) { \fontsize{8}{15}\selectfont$-L$ };
        \node[] at (4,-0.2) { \fontsize{8}{15}\selectfont$L$ };
        \node[] at (6.3,-0.2) { \fontsize{8}{15}\selectfont$2L$ };
        \node [black] at (2,0) {\textbullet };
        \node [black] at (2,-.2)  {\fontsize{8}{15}\selectfont $0$};
        \node [black] at (3,-.6)  {\fontsize{8}{15}\selectfont $\tilde Gw = 0  \, \downarrow$};
        \node [black] at (3,1.7)  {\fontsize{8}{15}\selectfont $\tilde Gw = 0\, \uparrow $};

        \path[->] (5.5,1.7) edge [bend left] node[above] {$H_\delta \tau_\delta $} (7.7,1.7);

        %
        \begin{scope}[shift={(8,-0.3)}]
          \fill[blue!10] (4,0)--(4,2.1-0.6)--(6,2.1-0.6)--(6,0)--cycle;
          \fill[blue!20] (2,0)--(2,2.1-0.6)--(4,2.1-0.6)--(4,0)--cycle; 
          \fill[blue!10] (0,0)--(0,2.1-0.6)--(2,2.1-0.6)--(2,0)--cycle;

          \draw[dashed] (6,2.1-0.3) -- (6,2.1) -- (0,2.1) --(0,2.1-0.3);
          \draw[dashed] (0,0.3) --(0,0) -- (6,0) -- (6,0.3);

          \draw [] (0,0+.3) --(0,1.5+.3)--(6,1.5+.3)--(6,0+.3)--cycle;
          \draw [ thick] (3,  1+.3) node [ below]{ \fontsize{10}{15}\selectfont $H_\delta \tau_\delta \tilde G w(x,t)$};
          \draw [->,dotted] (2, -0.5) -- (2, 2+.6) node[left]{  \fontsize{10}{15}\selectfont$t$};
          \draw [->,dotted] (-0.5, 0+.3) -- (7, 0+.3) node[below]{  \fontsize{10}{15}\selectfont$x$};
          \draw [black] (6,  2.7) node [ below]{ \fontsize{10}{14}\selectfont $\hat\om_\delta$};
          \node[] at (-.3,0.1) { \fontsize{8}{15}\selectfont$-L$ };
          \node[] at (4,0.1) { \fontsize{8}{15}\selectfont$L$ };
          \node[] at (6.3,0.1) { \fontsize{8}{15}\selectfont$2L$ };
          \node [black] at (2,0+.3) {\textbullet };
          \node [black] at (2,+.1) {\fontsize{8}{15}\selectfont $0$};

          \draw[->] (5,2.1+0.2)--(5,2.1);
          \draw[->] (5,2.1-0.5)--(5,2.1-0.3);
          \node at  (5,2.1-0.15) {\fontsize{7}{15}\selectfont$\delta$};
        \end{scope}
      \end{tikzpicture}
      \caption{Extension and translation in the proof of Theorem~\ref{thm:dense}.}
      \label{fig:density}
    \end{center}
  \end{figure}

  {\em {\underline{\smash[b]{Step~2. Translate:}}}} Let $\Gt v$ denote \rev{the}
  extension of $G v$ by zero to $\RRR^2$; i.e., $\Gt v$ equals $G v$
  on $\hat \om$ and equals zero elsewhere.  Let $\tau_\delta$ be the
  translation operator in the $-t$ direction by $\delta$; i.e.,
  $(\tau_\delta w) (x,t) = w(x, t + \delta).$ Its well known \cite{Brezi11} that
  \begin{equation}
    \label{eq:transcont}
    \lim_{\delta\to 0} \| \tau_\delta g - g \|_{\RRR^2}  = 0
    \qquad \forall g \in L^2(\RRR^2).
  \end{equation}
  Let $\hat \om_\delta = (-L,2L) \times (-\delta, T+\delta)$ and let 
  $H_\delta$ denote the restriction of functions on $\RRR^2$ to
  $\hat \om_\delta$.  By a change of variable,
  \begin{equation}
    \label{eq:tautrans}
    (\tau_\delta \Gt f, g)_{\hat \om_\delta} 
    = ( Gf, \tau_{-\delta}g)_{\hat\om}
    \qquad \forall f \in L^2(\om), \; g \in L^2(\hat\om_\delta).
  \end{equation}
  We now claim that
  \begin{equation}
    \label{eq:AtauG}
    A H_\delta \tau_\delta \Gt v = H_\delta \tau_\delta \Gt Av
    \qquad \forall v \in V^*.
  \end{equation}
  Indeed, for any $\varphi \in \D(\hat \om_\delta)$, the action of the
  distribution $A H_\delta\tau_\delta \Gt v$ on $\bar \varphi$ equals
  \begin{align*}
    \ip{ A H_\delta\tau_\delta \Gt v, \varphi}_{\D(\hat\om_\delta)}
    & = ( \tau_\delta \Gt v, A\varphi)_{\hat\om_\delta}
      = ( G v,   A \tau_{-\delta}\varphi)_{\hat\om}
      = (  v,    G'A \tau_{-\delta}\varphi)_{\om} 
      = (  v,    AG' \tau_{-\delta}\varphi)_{\om} 
    \\
    &      = (  Av,    G' \tau_{-\delta}\varphi)_{\om} 
      - \ip{ D v, G'\tau_{-\delta}\varphi}_W,
  \end{align*}
  where we have used~\eqref{eq:tautrans}, \eqref{eq:GG'},
  and~\eqref{eq:G'A} consecutively.  Since
  $(G' \tau_{-\delta} \varphi)|_\vG =0,$ Lemma~\ref{lem:VVsubset}
  shows that $G' \tau_{-\delta} \varphi$ is in $V$, and consequently
  the last term above vanishes for all $v \in V^*$.  Continuing and
  using~\eqref{eq:GG'} and~\eqref{eq:tautrans} once more,
  $ \ip{ A H_\delta \tau_\delta \Gt v, \varphi}_{\D(\hat\om_\delta)} =
  (GAv, \tau_{-\delta} \varphi)_{\hat \om} = (\tau_\delta \Gt Av,
  \varphi)_{\hat \om_\delta}.$ This proves~\eqref{eq:AtauG}.

  {\em {\underline{\smash[b]{Step~3. Mollify:}}}} Consider the
  mollifier $\rho_\veps\in \D(\RRR^2)$, for each $\veps>0$, defined by
  $ \rho_\veps(x,t) = \veps^{-2} \rho_1( \veps^{-1} x,\veps^{-1} t),$
  where
  \[
  \rho_1(x,t)
  =
  \left\{
    \begin{aligned}
      & k\, e^{-\frac{1}{1-x^2-t^2}}
      && \text{ if } x^2 + t^2<1,
      \\
      & 0
      && \text{ if } x^2 + t^2\geq 1,
    \end{aligned}
  \right. 
  \]
  and $k$ is a constant chosen so that $\int_{\RRR^2} \rho_1 =1.$
  It is well known \cite{Brezi11} that when any function $w$ in $L^2(\RRR^2)$ is
  convolved with $\rho_\veps$, the result $\rho_\veps * w$ is
  infinitely smooth and satisfies
  \begin{equation}
    \label{eq:mollL2}
    \lim_{\veps \to 0} \| w - \rho_\veps * w \|_{\RRR^2 } = 0
    \qquad 
    \forall \, w \in L^2(\RRR^2).
  \end{equation}
  Consider any small enough $\delta>0$, say $\delta < \min(L/2, T/2)$,
  and define two functions $ v_\veps = \rho_\veps * \tau_\delta \Gt v$
  and $ a_\veps = \rho_\veps * \tau_\delta \Gt A v.$ Note that the two
  smooth functions $A v_\veps$ and $a_\veps$ need not coincide
  everywhere. However, because of~\eqref{eq:AtauG}, they coincide on
  $\om$ whenever $\veps < \delta/2$:
  \[
  A v_\veps = a_\veps \qquad \text{ on } \om.
  \]
  Let us therefore set $\delta$ to, say, $\delta=3 \veps$ and let $\veps < 
  \min(L/2, T/2)/3$ go to zero. Note that
  \begin{align*}
    \| Av_\veps - Av \|_{\om}
    & = \| a_\veps - Av \|_\om = 
      \|  \rho_\veps * \tau_\delta \Gt A v - Av \|_\om
    \\
    & \le 
      \|  \rho_\veps * \tau_\delta \Gt A v -  
          \tau_\delta \Gt A v \|_{\RRR^2} 
      + \| \tau_\delta \Gt A v  - \Gt Av \|_{\RRR^2},
\\
    \| v_\veps - v \|_{\om}
    &
      \le 
      \| \rho_\veps * \tau_\delta \Gt v  - \tau_\delta \Gt v \|_{\om} 
      + 
      \| \tau_\delta \Gt v - v \|_{\om} 
    \\
    & \le 
      \| \rho_\veps * \tau_\delta \Gt v  - \tau_\delta \Gt v \|_{\RRR^2} 
      + 
      \| \tau_\delta \Gt v - \Gt v \|_{\RRR^2}.
  \end{align*}
  Using~\eqref{eq:mollL2} and~\eqref{eq:transcont}, it now immediately
  follows that
  \[
  \lim_{\veps \to 0} \| v_\veps - v \|_W =0.
  \]

  To conclude, examine the value of $v_\veps$ at
  points $z=(0,t)$ for any $0<t<T$, namely 
  \[
  v_\veps(0,t) = \int_{\RRR}\int_\RRR 
  \rho_\veps(-x', t -t') \; (\tau_\delta \Gt v)(x',t')\; 
  dx'\,dt'.
  \]
  The integrand of the inner integral is the product of an even
  function $(\rho_\veps)$ of $x'$ and an odd function
  $\tau_\delta \Gt v$ of $x'$. Hence $v_\veps(0,t)=0$. The same holds
  for points $z = (L,t)$.  Moreover, since $\tau_\delta \Gt v(y)$ is
  identically zero in a neighborhood of $z = (T,x)$ for all $0<x<L$,
  we conclude that $v_\veps|_{\vG^*}=0.$
\end{proof}

\section{Error Estimates for the ideal DPG method}\label{sec:error}

Continuing to consider the set  $\om$ as defined in Section~\ref{sec:dense}, we
now proceed to analyze the convergence of the ideal DPG method for
Problem~\ref{prb:weakSchrodinger}. The ideal DPG method finds $u_h$
and $q_h$ in finite-dimensional subspaces $U_h\subset L^2(\om)$ and
$Q_h \subset Q$, respectively, satisfying 
\begin{align}
b((u_h, q_h), v) = F(v) \qquad
  \forall  v \in T(U_h\times Q_h ). \label{discrete}
\end{align} 
Here $ T : L^2(\om)\times Q \to W_h$ is defined by
$( T(z,r), v ) = b((z,r),v)$ for all $v \in W_h$ and any
$(z,r) \in L^2(\om)\times Q$.  The main feature of the ideal DPG
method is that the well-posedness of Problem~\ref{prb:weakSchrodinger}
implies quasioptimality of the method's error~\cite{DemkoGopal11}.
The wellposedness of Problem~\ref{prb:weakSchrodinger} follows from
Theorem~\ref{thm:uw_schrod}, now that we have verified
Assumption~\ref{asm:dense} in Theorem~\ref{thm:dense}. Hence to obtain
convergence rates for specific subspaces, we need only develop
interpolation error estimates. Since the interpolation properties of
the $L^2$-conforming $U_h$ are standard, we need only discuss those of
$Q_h.$ To study this, we will create a spacetime finite element space
$V_h \subset V$, then identify $Q_h$ as $D_h(V_h)$, and finally
establish interpolation estimates for $Q_h$ using those for~$V_h$.
Note that $V_h$ will be used only in the proof (and not in the
computations).

To transparently present the ideas, we shall limit ourselves to the
very simple case of a uniform mesh $\oh$ of spacetime square elements
of side length~$h$.  Let $\eh$ denote the set of edges of $\oh$.  On
any $E \in \eh,$ let $P_p(E)$ denote the space of polynomials on the
edge of degree at most~$p$.  On any $K \in \oh$, let $Q_p(K)$ denote
the space of polynomials of degree at most $p$ in $x$ and at most $p$
in $t$.  To begin the finite element construction, we consider the
reference element $\hK = (0,1) \times (0,1)$ and the element space
$Q_p(\hK),$ endowed with the following degrees of freedom: 
\rev{ For any $w \in H^{3}(K)$}, and for each
$i \in \{ 0, 1,\dots, p-2\}$ and $j \in \{ 0,1,2,\dots, p\},$ write
$x_i = i/( p-2)$ and $t_j = j/p$ and set
\[
\sigma_{ij}(w) = w(x_i, t_j),
\qquad 
\sigma_j^0(w) = \d_x w(0,t_j), 
\qquad
\sigma_j^1(w) = \d_x w(1,t_j).
\]
Together, these form a set $\vSig$ with
$(p-1)(p+1) + 2 (p+1)$ linear functionals (see Figure~\ref{fig:element}).  The triple
$(\hK, Q_k(\hK), \vSig)$ is a unisolvent finite element, in the
sense of~\cite{ciarletbook}, as we show next.

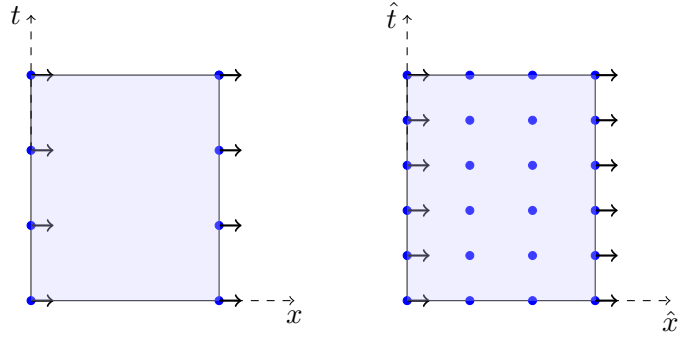
\begin{figure}
\begin{center}
\begin{tikzpicture}
\begin{scope}
\coordinate (top) at  (0,3);
\coordinate (bot) at  (2.5,0);
\coordinate (lft) at  (0,0);
\coordinate (rgt) at  (2.5,3);
\draw (lft)  -- (bot) ;
\draw (top) -- (rgt) ;
\draw  (rgt)  -- (bot) ;
\draw  (top)  --  (lft) ;
\filldraw[blue] (0,0) circle (1.5pt);
\filldraw[blue] (0,1) circle (1.5pt);
\filldraw[blue] (0,2) circle (1.5pt);
\filldraw[blue] (0,3) circle (1.5pt);
\filldraw[blue] (2.5,0) circle (1.5pt);
\filldraw[blue] (2.5,1) circle (1.5pt);
\filldraw[blue] (2.5,2) circle (1.5pt);
\filldraw[blue] (2.5,3) circle (1.5pt);
\filldraw[blue] (2.5,3) circle (1.5pt);

\draw[->,thick] (0,0) -- ($(.3,0)$) ;
\draw[->,thick] (0,1) -- ($(.3,1)$) node[left] {};
\draw[->,thick] (0,2) -- ($(.3,2)$) node[below] {};
\draw[->,thick] (0,3) -- ($(.3,3)$) node[left] {};

\draw[->,thick] (2.5,0) -- ($(2.5+ .3,0)$) ;
\draw[->,thick] (2.5,1) -- ($(2.5+ .3,1)$);
\draw[->,thick] (2.5,2) -- ($(2.5+.3,2)$);
\draw[->,thick] (2.5,3) -- ($(2.5+.3,3)$);

\fill [fill=blue!20, fill opacity=0.3] 
(bot) -- (rgt) -- (top) --(lft) ;
\draw[->,dashed] (2.5,0) -- ($(3.5,0)$) node[below] {$x$};
\draw[->,dashed] (0,2) -- ($(0,3.8)$) node[left] {$t$};
\end{scope}
\begin{scope}[shift={(5,0)}]
\coordinate (top) at  (0,3);
\coordinate (bot) at  (2.5,0);
\coordinate (lft) at  (0,0);
\coordinate (rgt) at  (2.5,3);
\draw (lft)  -- (bot) ;
\draw (top) -- (rgt) ;
\draw  (rgt)  -- (bot) ;
\draw  (top)  --  (lft) ;
\filldraw[blue] (0,0) circle (1.5pt);
\filldraw[blue] (0,3/5) circle (1.5pt);
\filldraw[blue] (0,2*3/5) circle (1.5pt);
\filldraw[blue] (0,3*3/5) circle (1.5pt);
\filldraw[blue] (0,4*3/5) circle (1.5pt);
\filldraw[blue] (0,5*3/5) circle (1.5pt);

\filldraw[blue] (2.5/3,0) circle (1.5pt);
\filldraw[blue] (2.5/3,3/5) circle (1.5pt);
\filldraw[blue] (2.5/3,2*3/5) circle (1.5pt);
\filldraw[blue] (2.5/3,3*3/5) circle (1.5pt);
\filldraw[blue] (2.5/3,4*3/5) circle (1.5pt);
\filldraw[blue] (2.5/3,5*3/5) circle (1.5pt);
\filldraw[blue] (2*2.5/3,0) circle (1.5pt);
\filldraw[blue] (2*2.5/3,3/5) circle (1.5pt);
\filldraw[blue] (2*2.5/3,2*3/5) circle (1.5pt);
\filldraw[blue] (2*2.5/3,3*3/5) circle (1.5pt);
\filldraw[blue] (2*2.5/3,4*3/5) circle (1.5pt);
\filldraw[blue] (2*2.5/3,5*3/5) circle (1.5pt);


\filldraw[blue] (2.5,0) circle (1.5pt);
\filldraw[blue] (2.5,3/5) circle (1.5pt);
\filldraw[blue] (2.5,2*3/5) circle (1.5pt);
\filldraw[blue] (2.5,3*3/5) circle (1.5pt);
\filldraw[blue] (2.5,4*3/5) circle (1.5pt);
\filldraw[blue] (2.5,5*3/5) circle (1.5pt);

\draw[->,thick] (0,0) -- ($(.3,0)$) ;
\draw[->,thick] (0,3/5) -- ($(.3,3/5)$) ;
\draw[->,thick] (0,2*3/5) -- ($(.3,2*3/5)$) ;
\draw[->,thick] (0,3*3/5) -- ($(.3,3*3/5)$) ;
\draw[->,thick] (0,4*3/5) -- ($(.3,4*3/5)$) ;
\draw[->,thick] (0,5*3/5) -- ($(.3,5*3/5)$) ;

\draw[->,thick] (2.5+.0,0) -- ($(2.5+.3,0)$) ;
\draw[->,thick] (2.5+0,3/5) -- ($(2.5+.3,3/5)$) ;
\draw[->,thick] (2.5+0,2*3/5) -- ($(2.5+.3,2*3/5)$) ;
\draw[->,thick] (2.5+0,3*3/5) -- ($(2.5+.3,3*3/5)$) ;
\draw[->,thick] (2.5+0,4*3/5) -- ($(2.5+.3,4*3/5)$) ;
\draw[->,thick] (2.5+0,5*3/5) -- ($(2.5+.3,5*3/5)$) ;

\fill [fill=blue!20, fill opacity=0.3] 
(bot) -- (rgt) -- (top) --(lft) ;
\draw[->,dashed] (2.5,0) -- ($(3.5,0)$) node[below] {$\hat x$};
\draw[->,dashed] (0,2) -- ($(0,3.8)$) node[left] {$\hat t$};
\end{scope}
\end{tikzpicture}
\caption{Degrees of freedom in the $p =3$ (left) and
  $p=5$ (right) cases.}
\label{fig:element}
\end{center}
\end{figure}

\begin{lemma} \label{lem:unisolv} %
  Suppose $p\ge 3$. Then any polynomial $w \in Q_{p}(\hK)$ is uniquely
  defined by the values of its degrees of freedom $ \sigma$ in $\vSig.$
\end{lemma}
\begin{proof}
  Suppose $w \in Q_p(\hK)$ and $\sigma(w)=0$ for all
  $\sigma \in \vSig$.  Then $w_j(x) = w(x, t_j)$ is a polynomial of
  degree $p$ in one variable ($x$). The Hermite and Lagrange degrees
  of freedom on $t=t_j$ imply $w_j=0$. Now, fixing~$x$, observe that the
  polynomial $w(x,t)$ is of  degree at
  most $p$ in the variable $t$ and has $p+1$ zeros.  Hence
  $w\equiv 0$ and the proof is complete since $\dim Q_p(\hK )$ equals
  the number of degrees of freedom.
\end{proof}

Next, consider the global finite element space
$W_h^p(\om) = \{ w \in L^2(\om) : \; \d_t w$ and $\d_{xx} w$ are in
$L^2(\om)$ and $w|_K \in Q_p(K) $ for all $K \in \oh\}.$ Each element
$K \in \oh$ is obtained by mapping the reference element $\hK$ by
$T_K : \hat K \to K,$
$T_K(\hat x,\hat t ) = ( h\hat x +x_K, h\hat t + t_K)$, where
$(x_K, t_K)$ is the lower left corner vertex of $K,$ and the element
space $Q_p(K)$ is the pull back of the reference element space
$Q_p(\hK)$ under this map.  The space $W_h^p(\om)$ can be controlled
by a global set of degrees of freedom obtained by mapping the
reference element degrees of freedom and, as usual, coalescing those
that coincide at the mesh element interfaces.

On the reference element $\hK$, the degrees of freedom define an 
interpolation operator 
\[
\vPih w =  
\sum_{\sigma \in \vSig}\, \sigma (w ) \,\varphi_\sigma,
\]
where, as usual, $\{ \varphi_\eta \in Q_p(\hK): \eta \in \vSig\}$ is the set of 
shape functions obtained as the dual basis of~$\vSig$.  By the Sobolev
inequality in two dimensions, $\vPih: H^3(\hK) \to Q_p(\hK)$ is
continuous. Similarly, the global degrees of freedom define an
interpolation operator $\vPi: H^3(\om) \to W_h^p(\om)$ satisfying
\begin{equation}
  \label{eq:commute}
  (\vPi w)\circ T_K  =  \vPih (w\circ T_K).
\end{equation}

\begin{lemma} \label{lem:interp_err} %
If $w \in H^{\rev{p+1}}(\om)$, then for all $p\ge 3$, 
\begin{align*}
\| w -\vPi w \|_\om &  \leq C h ^{p+1} | w | _{H^{p+1}(\rev{\om})}, \\
\|\dt ( w - \vPi w) \|_\om & \leq C h ^{p} |w |_{ H ^{p+1} (\rev{\om}) },\\
\|\dxx ( w - \vPi w) \|_\om & 
\leq C h ^{p-1} |w |_{ H ^{p+1} (\rev{\om})}.
\end{align*}
\end{lemma}
\begin{proof}
  Changing variables $(x,t) = T_K( \hat x, \hat t)$ as
  $(\hat x, \hat t)$ runs over $\hK,$ integrating, and using~\eqref{eq:commute},
  \begin{subequations}
    \label{eq:interpolations}
    \begin{align}
      \label{interpolations0}
      \| w - \vPi w \|_K & =h \| \hat w - \vPih \hat w \|_{\hat K } ,
      \\  
      \label{interpolations1}
      \| \dt ( w - \vPi w ) \|_K &  = \| \partial_ {\hat t }( \hat w - \vPih \hat w )   \|_{\hat K }  ,
      \\ 
      \label{interpolations}
      \| \dxx ( w - \vPi w ) \|_K
                         & = 
                           h^{-1}
                   \|\partial_{\hat x \hat x }( \hat w  - \vPih \hat w )  \|_{ \hat K}. 
    \end{align}    
  \end{subequations}
  On the reference element, since
  $ H^{p+1}(\hat K) \hookrightarrow H^3(\hat K),$ the interpolation
  operator $\vPih : H^{p+1}(\hK) \to Q_p(\hK)$ is continuous.
  Moreover $\vPih \hat w = \hat w $ for all $\hat w \in Q_p(\hK)$. Hence,
  the Bramble-Hilbert Lemma yields a $\hat C>0$ such that
  $\| \hat w - \vPih \hat w \|_{H^3(\hat K )} \leq \hat C | \hat w |_{
    H^{p+1} ( \hat K ) }$
  for all $\hat w \in H^{p+1}(\hK)$.  Since
  $| \hat w |_{H^{p+1}(\hat K ) } \leq C h^{p} |w | _ {H^{p+1}(K)},$
  combining with \eqref{eq:interpolations} and summing over all the
  elements in $\om_h$, we obtain the result.
\end{proof} 

Now we are ready to present the main result of this section. Set
$V_h = W_h^p(\om) \cap V$ and 
\begin{align}  \label{eq:UhQh}
  Q_h & = D_h (V_h), 
  & U_h &= \{ u \in L^2(\om) : u|_K \in Q_{p-1}(K) \text{ for all } K \in \om_h \}.
\end{align}

\begin{theorem} \label{thm:ee} %
  Let $p\ge 3$. Suppose $ u \in V\cap H^{\rev{p+1}}(\om)$ and $q = D_hu$ solve
  Problem~\ref{prb:weakSchrodinger} and suppose $U_h \times Q_h$ is
  set by~\eqref{eq:UhQh}.  Then, there exists a constant $C$ independent of
  $h$ such that the discrete solution $u_h \in U_h$ and $q_h \in Q_h$
  solving~\eqref{discrete} satisfies 
\begin{align}
\|u -u_h \|_\om + \|q-q_h \|_Q \leq C  h ^r |u |_{ H^{r+2}(\rev{\om})}
\end{align}
for $2 \leq  r \le  p-1 .$ 
\end{theorem} 
\begin{proof}
  By~\cite[Theorem~2.2]{DemkoGopal11} the ideal DPG method is quasioptimal:
  \begin{align*}
  \| (u,q) - (u_h,q_h) \|_{U \times Q}^2
    & \le \; C \inf_{ (z_h, r_h) \in U_h \times Q_h }  
      \| (u,q) - (z_h,r_h) \|_{U \times Q}^2
    \\
    & = C 
      \inf_{ (z_h, r_h) \in U_h \times Q_h }  \left( 
      \| u - z_h \|_\om^2 + \| q - r_h\|_Q^2 \right).
  \end{align*}
  Because of the standard approximation estimate
  $ \inf_{z_h \in U_h} \| u - z_h\|_\om \le Ch^r |u|_{H^r(\om)}$ for
  $0 \le r \le p-1,$ it suffices to focus on $\|q - r_h\|_Q$.  Since
  $q = D_h u,$ by the definition of $Q$-norm~\eqref{eq:Qnorm}, and the
  fact that any $r_h$ in $Q_h$ equals $D_h v_h$ for some
  $v_h \in V_h$, we have
  \begin{align*}
    \inf_{r_h \in Q_h} \| q - r_h \|_Q 
    & \le  \inf_{v_h \in V_h} \| u - v_h\|_W \le \| u - \vPi u\|_W.
  \end{align*}
  Applying Lemma~\ref{lem:interp_err}, the result follows.
\end{proof}

We conclude this section by examining a property of $Q_h$ that is
useful for computations.  Let $\eh^{\shortmid}$ and $\eh^{\mbox{-}}$
denote the set of vertical and horizontal (closed) mesh edges,
respectively, and $\eh^+ = \eh^\shortmid \cup \eh^{\mbox{-}}$. Let
$E_h^\shortmid$ and $E_h^+$ denote the closed set formed by the union
of all edges in $\eh^\shortmid$ and $\eh^+$, respectively.  Let
$Q_h^\shortmid =\{ r \in L^2(E_h^\shortmid): r|_F \in P_p(F)$ for all
$F \in \eh^\shortmid \}$ and $Q_h^+ =\{ r \in L^2(E_h^+) : r$ is continuous on
$E_h^+$ and $ r|_F \in P_p(F)$ for all $F \in \eh^+$ and
$r|_{\vG} =0 \}.$ For any $v_h \in V_h$, since $v_h$ is a polynomial
on each element, we may integrate by parts element by element to get
\begin{align*}
  \ip{D_h v_h, \psi}_h
  & = (A_h v_h, \psi)_h - (v_h, A_h \psi)_h \\
  & = \sum_{K\in \oh} \int_{\d K} i n_t v_h \bar \psi  
    + \int_{\d K} v_h  n_x (\d_x \bar \psi) 
    - \int_{\d K} n_x (\d_x v_h) \bar \psi
\end{align*}
for all $\psi \in \D(\bar \om)$. 
Thus $q = D_h v_h$ satisfies 
\begin{align*}
  \ip{q, \psi}_h
  & = \sum_{K\in \oh} \int_{\d K} q^+ (i n_t \bar \psi )
    + \int_{\d K} q^+ n_x (\d_x \bar \psi) 
    - \int_{\d K} q^\shortmid (n_x \bar \psi),
\end{align*}
where $q^+ = v_h|_{E_h^+}$ and
$q^\shortmid = \d_x v_h|_{E_h^\shortmid}.$ In computations, one may
therefore identify $Q_h$ with the interfacial polynomial space
$Q_h^+ \times Q_h^\shortmid$ whose components are of degree at most~$p$.

\section{Numerical Results}\label{sec:results}

This section is motivated by our interest in simulating
electromagnetic pulse propagation in dispersive optical fibers.
Nonlinear, dispersive Maxwell equations in the context of optical
fibers have been studied extensively~\cite{agrawal1}. The common
approach to model dispersive, intensity-dependent nonlinearities is
based on several simplifying approximations. These approximations
include a slowly varying pulse envelope, a quasi-monochromatic optical
field, a specific polarization maintained along the fiber length, and
approximation of nonlinear terms as perturbations of the purely linear
case. With these assumptions, the full Maxwell equations are
reduced~\cite{agrawal1,shaw04} to the ``nonlinear Schr{\"o}dinger
equation''
\[
i \frac{\partial a}{\partial x} - \frac{\beta}{2} \frac{\partial^2 a}{\partial t^2} + \gamma |a|^2 a = 0,
\]
where $x$ is the distance along the fiber, $t$ is an observation
window (in time), $\beta$ is some given fiber-dependent constant, and
$a$ is a complexified amplitude of the pulse.

Since the roles of $x$ and $t$ in this application may be
confusing, we switch them to agree with the previous sections and
consider the simple case of $\gamma=0.$ 
In other words, we present numerical
results obtained using a practical DPG method applied to the one
dimensional Schr\"odinger problem
\begin{align*}
i{\partial _t}u - \frac{\beta}{2}{\partial_{xx}}u &= f, \qquad 0<x<1, 0<t<1,\\
u(x,0) & = u_0 (x), \qquad 0<x<1, \\
u(0,t) & = u_l(t), \qquad 0 <t<1, \\
u(1,t) &  = u_r(t), \qquad 0<t<1.
\end{align*}

To describe the method we used in practice, first set $ b( (u,q), v)$
using the Schr\"odinger operator $A = i \d_t - (\beta/2) \d_{xx}$ and
recall~(\ref{discrete}).  As mentioned in the previous section, the
action of any $q=D_h u$ on the boundary of each element, can be viewed
as a combination of two \emph{independent} boundary actions of
variables $q^+$ and $q^\shortmid$ that are of the same polynomial
order.  However, we are led to implement a slightly different space
because our computational tool is a standard Petrov-Galerkin code
supporting the exact sequence elements of the first type~\cite{oesf}.
Accordingly, $q^+$ is discretized with (continuous) traces of $H^1$
conforming elements of order $p$ but $q^\shortmid$ is discretized with
(discontinuous) traces of the compatible $H(\text{div})$ conforming
elements of order $p-1$, i.e., one order less than required by the
presented interpolation theory. Let $\tilde Q_h$ represent this
reduced space. The next modification needed in our implementation
is an approximation of~$T$. Let $T_h(z,r) \in W_h^{\Delta p}$ be defined by
$( T_h(z,r), v ) = b((z,r),v)$ for all $v \in W_h^{\Delta p}$ and any
$(z,r) \in U_h \times \tilde Q_h$, where
$W_h^{\Delta p} = \{ w \in W_h: w|_K \in Q_{p+\Delta p}(K)$ for all
$K\in \oh\}$. Thus the practically implemented method, in contrast
to~(\ref{discrete}), finds $u_h \in U_h$ and $q_h\in \tilde Q_h$
satisfying
\[
b((u_h, q_h), v) = F(v) 
\] 
for all $v \in T_h(U_h\times \tilde Q_h )$.  In all the results
presented below, no significant differences were seen between $\Delta
p =1$ and $\Delta p=2$, so we only report the results 
obtained with $\Delta p=1$.

\begin{figure}
\includegraphics[width=6cm, trim={0 5.5cm 0 3cm}, clip]{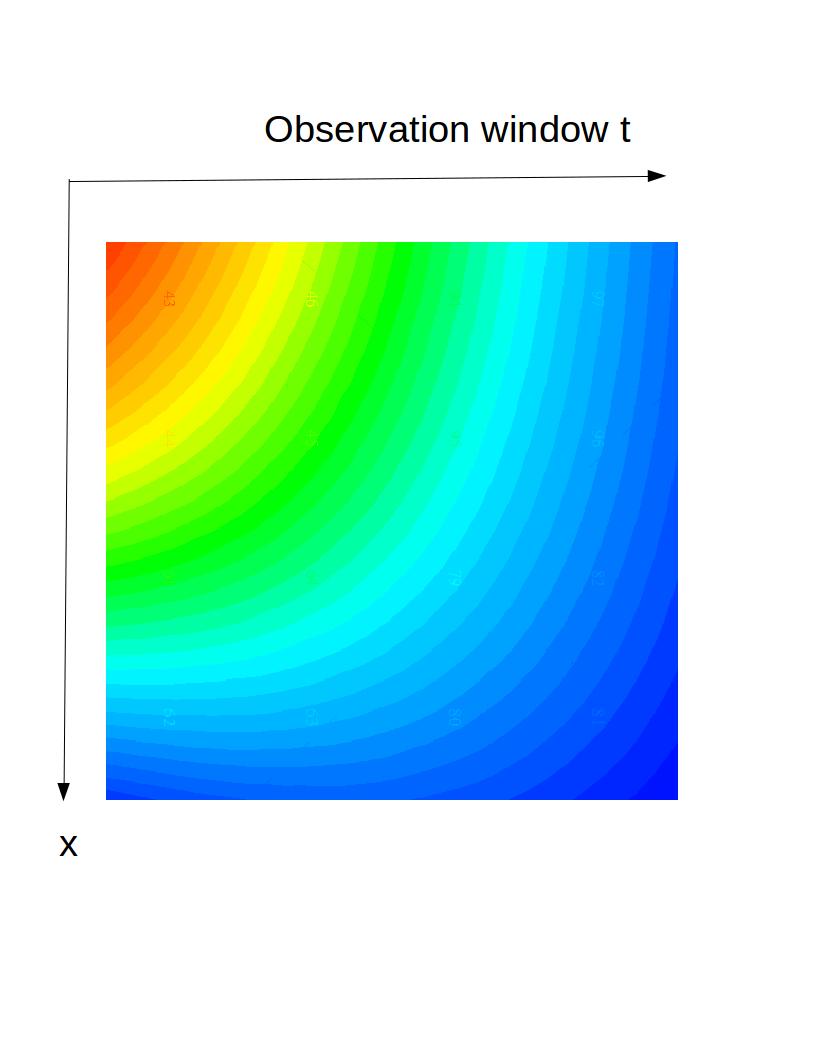}
\includegraphics[width=6cm, trim={0 5.5cm 0 3cm}, clip]{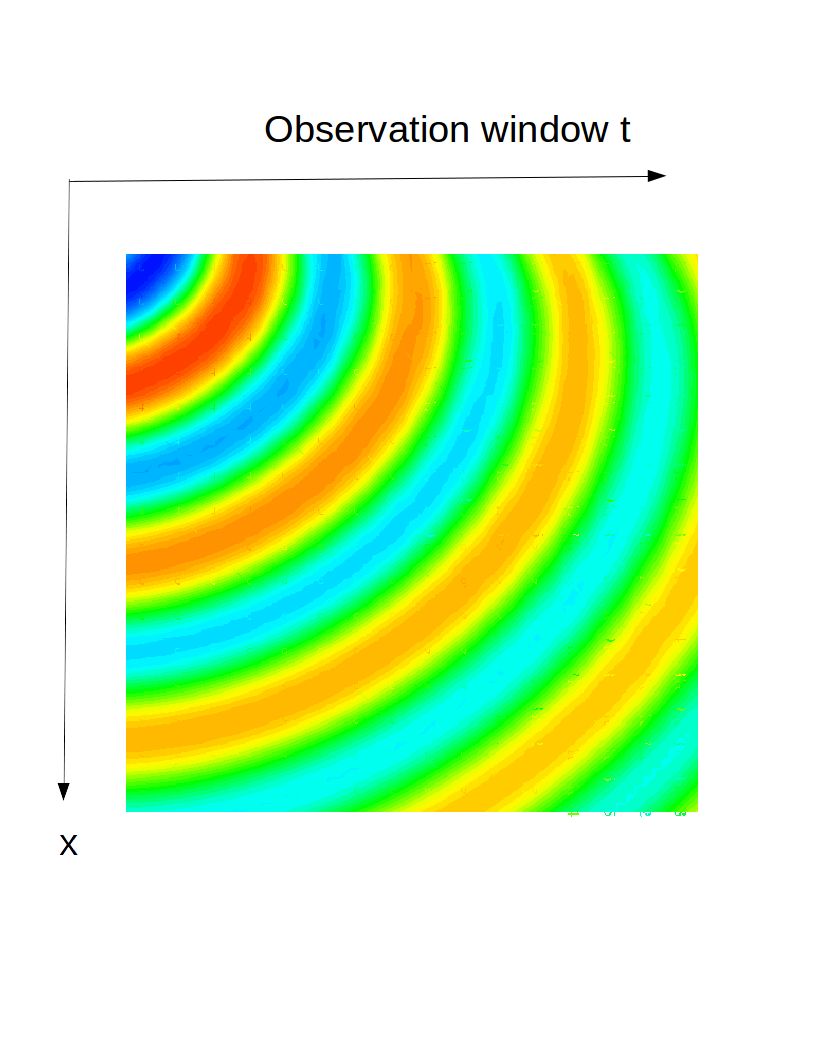}
\caption{Plots of solutions. {\em Left:} Case {\em (a)}. {\em Right:}
  Case {\em (b)}.}
\label{fig:solution}
\end{figure}

\begin{figure}[t]
\includegraphics[height=5.5cm]{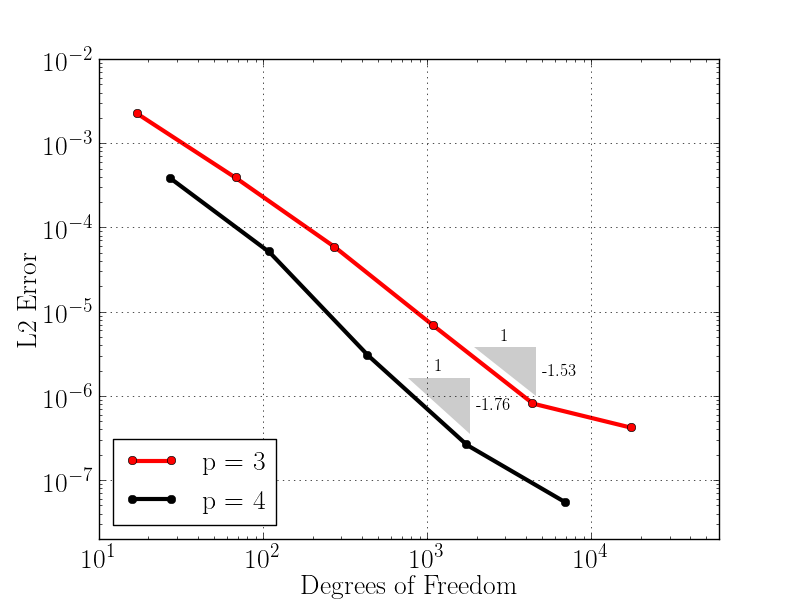}
\includegraphics[height=5.5cm]{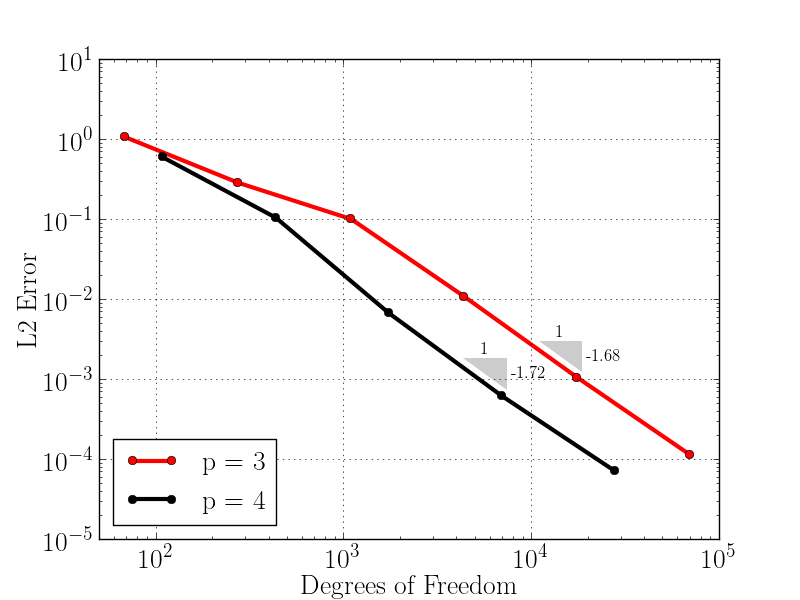}
\caption{Rates of convergence. {\em 
Left:} Results from case {\em (a)} with the 
complex Gaussian solution. {\em Right:} Results from the case {\em
  (b)} with the wave packet of frequency $\og=20$.}
\label{fig:rates}
\end{figure}

We report the observed rates of convergence for two problems: {\em
  (a)}~The first case is when the exact solution is a complex Gaussian
\begin{equation}\label{eq:leftprob}
u(x,t) = \frac{M T_0}{\sqrt{T_{0}^{2}-i \beta t }}e^{-\frac{x^2}{T_{0}^{2}-i \beta t }},
\end{equation}
where $M,T_0, $ and $\beta$ are fiber-dependent constants (see
\cite{shaw04}). Our simulations used non-dimensionalized units of
$M=T_0=1.5$ and $\beta = 2.5$.  {\em (b)}~The second example uses an
exact solution which is a wave packet traveling along the fiber whose
in-packet oscillations are of moderately high wavenumber~$\og$, namely
\begin{equation}
u(x,t) = a_{0} e^{-\frac{x^2+t^2}{\omega^{2}}}, 
\end{equation}
where the amplitude $a_0 = (2/\og^2)^{1/4}
$
and the wavenumber is $\omega = 20$.
Plots of solutions in either case are displayed in
Figure~\ref{fig:solution}.  

The observed convergence rates are
displayed in Figure~\ref{fig:rates} (the left plot shows results from
case~{\em (a)} and the right plot shows results from case~{\em(b)}.
We experiment with $p=3$ and $p =4$ cases.  For the ideal DPG method
using the $U_h \times Q_h$ in~\eqref{eq:UhQh}, Theorem~\ref{thm:ee}
implies that the convergence rate in terms of the number of degrees of
freedom $n = O(h^{-2})$ is $O(n^{-s})$ where $s=(p-1)/2$. We observe
from Figure~\ref{fig:rates} that in spite of reducing $Q_h$ to
$\tilde Q_h$ and in spite of approximating $T$ by $T_h$, we continue
to observe a rate higher than~$s$. Namely, in the $p=3$ case, while we
expected a rate of $s \le 1$, the observed rate is $s\approx 1.5$. In
the $p=4$ case, while the expected rate is $s \le 1.5$, the observed
rate is between $1.5$ and $2$. An improved error analysis explaining
these observations is yet to be found.

Note the flattening out of the curves in left plot of
Figure~\ref{fig:rates}.  This is due to conditioning issues.  As with
any method using second order derivatives, we should be wary of
conditioning.  Indeed, the DPG system with $p=3$ or $p=4$, after 4 or
5 uniform refinements, has a condition number in the vicinity of
$O(10^{10})$. Therefore, the roundoff effect becomes apparent after we
achieve an error threshold around $10^{-6}$ or $10^{-7}$. This is the
cause of convergence curves flattening out in case {\em (a)}.  In case
{\em (b)}, due to $\omega =20$, we start with a higher error so the
loss of digits due to conditioning issues is postponed.

\appendix
\section{Abstract weak formulation}
\label{apn:weak}

In this section, we consider a boundary value problem involving a
general partial differential operator. We derive a mesh-dependent weak
formulation of the boundary value problem and show that it is possible
to identify sufficient conditions for its wellposedness. This section
can be read independently of the remainder of the paper.

Let $\om \subseteq \RRR^d$ be a bounded open set in $d \ge 1$
dimensions and let $k, l, m\ge 1$ be integers.  
Let \rev{$A$ be a differential operator} such that 
$i$th component of $A u$ is
  \begin{equation}
    \label{eq:A1}\tag{A-a}
    [A u]_i = \sum_{j=1}^m \sum_{|\alpha|\le k} \d^\alpha ( a_{ij\alpha} u_j),    
  \end{equation}
  where $a_{ij\alpha}: \om \to \CCC$ are functions for all
  $i=1,\ldots, l$, $j=1,\ldots, m,$ and all multi-indices
  $\alpha=(\alpha_1,\ldots \alpha_d)$ whose length is
  $|\alpha| = \alpha_1 + \cdots + \alpha_d \le k$. As
  usual, $\d^\alpha = \d_1^{\alpha_1} \cdots \d_d^{\alpha_d}$. 
  \rev{The formal adjoint of $A$ is 
    given by
    \begin{equation}
      \label{eq:A*}
[A^* v ]_j = \sum_{i=1}^l
\sum_{|\alpha| \le k} (-1)^{|\alpha|} \overline{a_{ji\alpha}}\, \d^\alpha  v_j .      
    \end{equation}
}
We assume that the coefficients $a_{ij\alpha}$ are such that 
\begin{equation}
  \label{eq:A3}\tag{A-b}
  A^*u \in \D'(\om)^m \quad  
  \text{ for all } u \in L^2(\om)^l.
\end{equation}
For example,~\eqref{eq:A3} is satisfied if 
$a_{ij\alpha}$ \rev{are smooth.}

For any nonempty open subset $S \subseteq \om,$ define the space 
\begin{equation}
  \label{eq:W}
   W(S) = \{ u \in L^2(S)^m: A u \in L^2(S)^l \},
\end{equation}
normed by
$ \| u \|_{W(S)} = \left( \| u \|_S^2 + \| A u \|_S^2 \right)^{1/2},$
and the space 
\begin{equation}
  \label{eq:W*}
W^*(S) = \{ u \in L^2(S)^l: A^* u \in L^2(S)^m \},  
\end{equation}
normed by
$ \| u \|_{W^*(S)} = \left( \| u \|_S^2 + \| A^* u \|_S^2
\right)^{1/2}.$
\rev{Here and throughout,
$(\cdot,\cdot)_\om$ and $\| \cdot \|_\om$ denote the inner product
and the norm, respectively, in $L^2(\om)$ or its Cartesian
products. } 
To simplify notation, we abbreviate $W = W(\om)$, $W^* = W^*(\om)$.
Clearly these are inner product spaces.

\begin{lemma}
The spaces $W(S)$ and $W^*(S)$ are  Hilbert spaces.
\end{lemma}
\begin{proof}
  Since the proofs for $W(S)$ and $W^*(S)$ are similar, we only show
  the first.  Suppose $u_n$ is a Cauchy sequence in $W\rev{(S)}$. Then $u_n$ is
  Cauchy in $L^2(S)^m$ and $A u_n$ is Cauchy in $L^2(S)^l$. Hence
  there is a $u \in L^2(S)^m$ and $f \in L^2(S)^l$ such that
  $\| u - u_n \|_S \to 0$ and $\| f - A u_n \|_S \to 0$.
  We will show that $u$ is in $W(S).$
  
  Let $\phi \in \D(S)^l$.  For each $w \in L^2(S)^m$, the
  distributional action of $Aw$ on $\bar \phi$, denoted by
  $\ip{ Aw, \phi}_{\D(S)^l},$ equals $(w, A^* \phi)_S$. If $w$ is
  also in $W(S)$, then
  \begin{equation}
    \label{eq:Aw}
    (A w, \phi)_S = \ip{ Aw, \phi}_{\D(S)^l} =(w, A^* \phi)_S
  \end{equation}
  for all $\phi$ in $\D(S)^l$.  To complete the proof, we
  apply~\eqref{eq:Aw} with $w = u_n$ to get
  \begin{align*}
    \ip{ Au, \phi}_{\D(S)^l}
    & = (u, A^*\phi)_S = 
      \lim_{n\to \infty} (u_n, A^* \phi)_S
      = 
      \lim_{n\to \infty} (Au_n, \phi)_S = (f, \phi)_S
  \end{align*}
  for all $\phi$ in $\D(S)^l$. Hence $Au =f $,
  and $u$ is in $W(S)$.
\end{proof}

Next, define bounded linear operators
\rev{$D_S: W(S) \to W^*(S)'$ and $D^*_S: W^*(S) \to W(S)'$ by
\begin{align}
  \label{eq:D}
  \ip{ D_Sw, \wt}_{W^*(S)} &= (Aw, \wt)_S - (w, A^*\wt)_S,
  \\\label{eq:D*}
  \ip{ D_S^*\wt, w}_{W(S)} & = (A^* \wt, w)_S - (\wt, A w)_S,
\end{align}
for all $w \in W(S)$ and $\wt \in W^*(S)$. When $S=\om$, we 
abbreviate $D_S$ and $D^*_S$ to $D$ and $D^*$, respectively.}  Here, like in
\eqref{eq:Aw}, and in the remainder, we use $\ip{\cdot,\cdot}_X$ to
denote the action of a linear functional in $X'$ on an element of $X$.

\rev{Next we will view $A : \dom(A) \subset L^2(\om)^m \to L^2(\om)^l$
as an unbounded 
linear operator, whose domain (denoted by $\dom(A)$)} is chosen so that 
\begin{equation}
  \label{eq:A2}\tag{A-c}
  \D(\om)^m \subseteq \dom(A). 
\end{equation}
This implies that $A$ is a densely defined operator.  Then,
identifying the dual of (Cartesian products of) $L^2(\om)$ with
itself, recall that the adjoint
\rev{$A^*: \dom(A^*) \subset L^2(\om)^l \to L^2(\om)^m$} is a uniquely
defined (unbounded) closed linear operator~\cite{Brezi11,OdenDemko10}
on $ \dom(A^*) = \{ s \in L^2(\om)^l:$ there is an
$\ell \in L^2(\om)^m $ such that $(Av, s)_\om = (v, \ell)_\om$ for all
$v \in \dom(A)\},$ satisfying $ (Av, \vt)_\om = (v, A^* \vt)_\om$ for
all $ v \in \dom(A)$ and $ \; \vt \in \dom(A^*).$ Note that
$\D(\om)^l \subseteq \dom(A^*)$.  \rev{Note also that by an abuse of
  notation, we have used $A^*$ to denote both the differential
  operator in~\eqref{eq:A*} and the adjoint operator of the unbounded
  $A$.}

When $\dom (A)$ is endowed with the topology of $W(\om)$, we call it
$V$, i.e., although $V$ and $\dom(A) $ coincide as sets, $V$ has the
topology of $W(\om)$ and $\dom(A)$ has the topology of
$L^2(\om)^m$. Similarly, $\dom(A^*)$ is called $V^*$ when it is
endowed with the topology of $W^*(\om)$. 
For the next result, recall that 
the left annihilator of any subspace $R$ of the dual space $X'$ of any Banach
space $X$ is defined by 
$ \lprp R = \{ w\in X: \ip{ s',w}_X=0\text{ for all }s'\in R\}.$

\begin{lemma} \label{lem:V1}
  \rev{In the setting of~\eqref{eq:A1}, \eqref{eq:A3} and \eqref{eq:A2}, we have} 
  $
  V^* = \lprp{D(V)}.
  $
\end{lemma}
\begin{proof}
  According to the (above-mentioned) definition of $\dom(A^*)$,
  for any $\vt \in \dom(A^*) = V^*$, there is an
  $\ell \in L^2(\om)^l$ such that 
  \begin{equation}
    \label{eq:2}
    (Av,\vt)_\om = (v,\ell)_\om, \qquad \forall v \in V.
  \end{equation}
  Due to~\eqref{eq:A2}, we may choose $v$ in
  $\D(\om)^m$. By~\eqref{eq:A3}, $A^*\vt$ is a distribution and
  by~\eqref{eq:2} this distribution is in $L^2(\om)^m$ and equals
  $\ell.$ In particular, $\vt$ is in $W^*(\om)$. Hence~\eqref{eq:D} is
  applicable, and in combination with~\eqref{eq:2} yields
  $ \ip{ D v, \vt}_{W^*(\om)} = (Av, \vt)_\om - (v, \ell)_\om = 0 $
  for all $v \in V$.  Hence $\vt \in \lprp{ D(V)}$ and we have proved
  that $V^* \subseteq \lprp{D(V)}$. The reverse containment is also
  easy to prove.
\end{proof}

We are interested in the boundary value problem of finding $u$
satisfying
\begin{equation}
  \label{eq:bvp}
  A u  = f,\qquad
  u  \in  \,V,
\end{equation}
given $f \in L^2(\om)^l$.  Homogeneous boundary conditions are
incorporated in $V.$ Consider the scenario where $\om$ is partitioned
into a mesh $\oh$ of finitely many open elements $K$ such that
$\bar \om = \cup_{K \in \oh} \bar K$. Here the index $h$ denotes
$\max_{K \in \oh} \diam(K)$.  \rev{Recall} $D_K$ and $D_K^*$ by replacing
\rev{$S$} by $K$ in~\eqref{eq:D} and~\eqref{eq:D*}.  
Additionally, set
\begin{align}
\label{eq:WhWh*}
W_h&=\prod_{K\in \oh} W(K),
&
(W_h^*)'&=\prod_{K\in\oh} W^*(K)'.
\end{align}
The spaces $W_h^*$
and $W_h'$ are defined similarly. The component on an element $K$ of
functions in such product spaces are indicated by placing $K$ as
subscript, e.g., for any $w$ in $W_h$, the component of $w$ on element
$K$ is denoted by $w_K$.  Let $D_h: W_h \to (W_h^*)'$ be the
continuous linear operator defined by
\[
\ip{ D_h w, v }_{W_h^*} = \sum_{K\in\oh} \ip{ D_K w_K, v_K}_{W^*(K)}
\]
for all $w \in W_h$ and $v \in W_h^*$.  To simplify notation, we
abbreviate
$\ip{D_h w, v}_{W_h^*}$ to $\ip{D_h w, v}_h$, i.e., duality pairing in
$W_h^*$ is simply denoted by $\ip{ \cdot, \cdot}_h$.  For any
$w \in W_h$, we denote by $A_h w$ the function obtained by applying $A$
to $w_K$, element by element, for all $K \in \oh$. The resulting
function $A_h w$ is an element of $\Pi_{K\in \oh}L^2(K)^l$, which is identified to be the same as $L^2(\om)^l$.  The operator
$A_h^*: W_h^* \to L^2(\om)^m$ is defined similarly. Thus
\begin{equation}
  \label{eq:Dh}
  \ip{ D_h w, v }_h = (A_h w, v)_\om - (w, A_h^* v)_\om
\end{equation}
for all $w \in W_h$ and $v \in W_h^*$. 

\begin{lemma}
  \label{lem:1}
  For all $w\in W$ and $v \in W^*$,  we have
  $
  \ip{ D_h w,v}_h = \ip{ D w, v}_{W}.
  $
\end{lemma}
\begin{proof}
  If $w \in W$ and $\rev{v} \in W^*$, then $ A_h w = Aw$ and
  $A_h^* \rev{v} = A^* \rev{v}$. Using this in~\eqref{eq:Dh},
  $    \ip{ D_h w,v}_h
  = (Aw,v)_\om - (w, A^* v)_\om
      = \ip{D w,v}_{W^*}$
whenever \rev{$w$ is in $W$ and $v$ is in $W^*$.}
\end{proof}

To derive the mesh-dependent weak formulation, multiply~\eqref{eq:bvp}
by a test function $v \in W_h$ and apply the definition of
$D_K$. Summing over all $K \in \oh$, we obtain
$(u, A^*_h v )_\om + \ip{ D_h u, v}_h = (f,v)_\om$ for all $v$ in
$W_h^*$.  
Let  
\begin{equation}
  \label{eq:Q}
  Q = \{ r \in (W_h^*)': \; \text{ there is a $v \in V$ such that }
  r = D_h v\}.   
\end{equation}
Setting $D_h u$ to be a new unknown $q$ in $Q$, we have thus
derived the following weak formulation with $F(v) = (f, v)_\om$.

\begin{problem} \label{prb:weak}{\em 
  Given  any $F \in (W_h^*)',$\,  find $u\in L^2(\om)^m$
  and $q \in Q$ such that
  \[
    \label{eq:prb}
  (u, A_h^* v)_\om + \ip{ q,v }_h = F(v),
  \qquad
  \forall v \in W_h^*.    
  \]}
\end{problem}

\begin{theorem}
  \label{thm:wellposed-mesh}
  In the setting of~\eqref{eq:A1}, \eqref{eq:A3} and \eqref{eq:A2}, suppose 
  \begin{align}
    \label{eq:asm:1}
    & V  = \lprp{D^*(V^*)}, \text { and } 
    \\
    \label{eq:asm:2} 
    &A: V \to L^2(\om)^l \text { is a bijection}.
  \end{align}
  Then, Problem~\ref{prb:weak} is well posed.  Moreover, if
  $F(v) = (f,v)_\om$ for some $f \in L^2(\om)^l,$ then the
  unique solution~$u$ of Problem~\ref{prb:weak} is in~$V,$ 
  solves~\eqref{eq:bvp}, and satisfies $q= D_hu$.
\end{theorem}

Before we prove this theorem, we must note how our assumptions allow a
natural topology on $Q$.  Specifically, \eqref{eq:asm:1} implies that
$V$ is a closed subspace of $W$. It is also a closed subspace of~$W_h$
since $W$ is continuously embedded in~$W_h$.  The same embedding also
shows that the restriction of $D_h$ to $V$, denoted by
$D_{h,V} : V \to (W_h^*)'$, is continuous.  Note that $Q$ is the range
of $D_{h,V}$.  For any $r$ in $Q$, we use $D_{h,V}^{-1} (\{ r \})$ to
denote the pre-image of $r$, i.e., set of all $v \in V$ such that
$r = D_h v.$ The continuity of $D_{h,V}$ implies that
$D_{h,V}^{-1}( \{ 0\} )$ is a closed subspace of $V$.  Hence
\begin{equation}
  \label{eq:Qnorm}
\| q\|_Q = \inf_{ v \in D_{h,V}^{-1}(\{ q \})  } \| v \|_W
\end{equation}
is a norm on $Q$. This quotient norm makes $Q$ complete. The
wellposedness result of Theorem~\ref{thm:wellposed-mesh} is to be
understood with $Q$ endowed with this norm.

\subsection{A proof of wellposedness}

We now give a proof of Theorem~\ref{thm:wellposed-mesh}.  Recall that
the right annihilator of any subspace $S \subseteq X$ is defined by
$ \rprp S = \{ w'\in W': \ip{ w',s}_W=0\text{ for all } s\in S\}.$ The
next lemma is used below to prove uniqueness.

\begin{lemma}
  \label{lem:DhV}
  If~\eqref{eq:asm:1} holds, then 
  $
  D_h V \subseteq \rprp{ (V^*)}.
  $
\end{lemma}
\begin{proof}
  Let $w \in V \subseteq W_h$.  Then for any $\vt \in V^*$, the
  functional $D_h w \in (W_h^*)'$ satisfies
  $\ip{ D_hw , \vt}_h = \ip{ D w, \vt}_{W^*}$ by Lemma~\ref{lem:1}. But
  $\ip{ D w, \vt}_{W^*}= -\overline{\ip{D^* \vt, w}}_W = 0$ 
  since \eqref{eq:asm:1} shows that 
  $w \in \lprp{D^*(V^*)}$. Hence $D_h w \in \Vsp$.
\end{proof}

\begin{proof}[Proof of Theorem~\ref{thm:wellposed-mesh}]
  We verify the uniqueness and inf-sup conditions of the Babu\v{s}ka
  theory to obtain wellposedness.
  To verify the uniqueness condition, we must prove that if
  \begin{equation}
    \label{eq:12}
    (u, A_h^* v)_{\om} + \ip{ q, v}_h = 0, \qquad \forall v \in W_h^*,
  \end{equation}
  then $u$ and $q$ vanishes.  Since $q = D_h z$ for some $z \in V$, by
  virtue of Lemma~\ref{lem:DhV}, $\ip{ q, v}_h =0$ for any $v$ in
  $V^*$. Hence~\eqref{eq:12} implies
  \begin{equation}
    \label{eq:5}
    (u, A_h^* v)_\om=0, \qquad \forall v \in V^*.
  \end{equation}
  In particular, since $\D(\om)^l \subseteq V^*$, this implies that
  $Au =0$ and therefore $u \in W$. Hence~\eqref{eq:D} and~\eqref{eq:5}
  imply $\ip{ Du, v}_{W^*} = 0$, or equivalently $\ip{ D^* v,u}_W=0$
  for all $v \in V^*$. Thus $u\in \lprp{D^*(V^*)}=V$. 
  The bijectivity of
  $A: V \to L^2(\om)^l $ then implies that $u=0$.  Returning
  to~\eqref{eq:12} and setting $u=0$, we see that
  $\ip{ q,v}_h=0$ for all $v\in W_h^*$, so $q=0$ as well.
  
  It only remains to prove the inf-sup condition 
  \begin{equation}
    \label{eq:inf-sup}
    \| w \|_{W_h^*} 
    \le
    C_1 \sup_{ 0 \ne x\in X } 
    \frac{ |b(x,  w) |}
    {\;\|x\|_X},
  \end{equation}
  where $X = L^2(\om)^m \times Q$ and
  $b((u,q), w) = (u, A_h^* w)_\om + \ip{ q,w }_h$.  Given any
  $w \in W_h^* \subseteq L^2(\om)^l$, we use the bijectivity of
  $A: V \to L^2(\om)^l$ and the Banach Open Mapping theorem 
  to obtain a $v$ in $V$ satisfying $ A v = w,$
  and $\| v \|_W \le C \| w \|_{W_h^*}.$ Then, setting $z = v + A_h^* w$ and
  $r = D_h v,$ we have $\| r \|_Q \le \| v \|_W\le C \| w \|_{W_h^*}$ and
  $\| z\|_\om \le (C+1) \| w \|_{W_h^*}$. Hence
  \begin{align*}
    \| w \|_{W_h^*}^2 
    &= ( A_h v, w)_\om + (A^*_h w,A^*_h w)_\om
      = (v+ A^*_h w, A^*_h w)_\om + \ip{ D_h v,w}_h
    \\
    &  = 
      \| (z, r)\|_{X}\;
      \frac{ b((z, r), \, w) }
      { \| (z,  r)\|_{X}  }
      \;\le
      \;C_1
      \| w\|_{W_h^*} \;
      \sup_{ 0 \ne x\in X } 
      \frac{ |b(x,  w) |}
      {\;\|x\|_X},
  \end{align*}
  where $C_1$ depends only on $C$. Hence~\eqref{eq:inf-sup} follows.
\end{proof}

\begin{remark}
  Various elements of the arguments used in this proof are well-known
  in the DPG literature -- see e.g.,
  \cite[\S~6.2]{DemkoGopalMuga11a}. A generalization of these ideas to
  make a unified theory for DPG approximations of all Friedrichs
  systems was attempted in~\cite{Bui-TDemkoGhatt13}.  However,
  \cite[equation~(2.17)]{Bui-TDemkoGhatt13} is not correct (a
  counterexample is easily furnished by the Laplace example) and
  unfortunately that equation is used in \cite[Lemma~2.4 and Corollary
  2.5]{Bui-TDemkoGhatt13} to prove the existence of a solution for
  Problem~\ref{prb:weak}. The above proof provides a corrigendum
  to~\cite{Bui-TDemkoGhatt13} and shows that the results claimed there
  for symmetric Friedrichs systems are indeed correct for operators of
  the form~\eqref{eq:A1} with $k=1$ and with $V$ and $V^*$ set
  respectively to the null spaces of the operators $B-M$ and $B+M^*$
  defined there.
\end{remark}
\begin{remark}
  The above analysis is applicable beyond Friedrichs systems as the
  example of Schr\"odinger equation shows.  ``Instead of working with
  one equation of higher than first order,'' writes Friedrichs in his
  early work~\cite{friedrich1}, ``we prefer to work with a system of
  equations of first order.''  We have already noted the difficulties
  in reformulating the Schr\"odinger equation as a first order system.
  The modern theory of Friedrichs systems (for operators of the
  form~\eqref{eq:A1} with $l=m$) starts with the assumption that
  $ \| (A + A^*) \phi \|_\om \le C \| \phi\|_\om$ for all
  $\phi \in \D(\om)^l$ -- see \cite[equation (T2)]{guermond1}.
  This assumption does not hold for the Schr\"odinger operator.
\end{remark}

\subsection{An alternate proof of wellposedness}

Another proof of Theorem~\ref{thm:wellposed-mesh} can be given using
the following two lemmas.

\begin{lemma}
  \label{lem:unbroken}
  $
  V^* = \{ y \in W_h^* : \; \ip{q, y}_h=0 \text{ for all } q \in
  Q\}.
  $
\end{lemma}
\begin{proof}
  If $y \in V^*,$ then for any $z \in V,$ using Lemmas~\ref{lem:1}
  and~\ref{lem:V1}, we have $\ip{D_h z, y}_h  = \ip{D z, y}_{W^*} =0,$ 
  i.e., $\ip{q, y}_h=0$ for all $q \in Q.$

  To prove the reverse containment, let $y \in W_h^*$ satisfy
  $\ip{D_h z, y}_h=0$ for all $z \in V.$ For any $\phi \in \D(\om)^m$,
  the distribution $A^* y$ satisfies
  $ \ip{ A^* y, \phi}_{\D(\om)^m} = ( y, A \phi)_\om = (A_h^* y,
  \phi)_\om + \overline{\ip{D_h \phi, y}}_h$.
  The last term is zero, because by~\eqref{eq:A2},
  $\D(\om)^m \subseteq \dom(A) = V$. Hence $A^* y = A_h^* y$ and $y$
  is in $W^*$. Thus by Lemma~\ref{lem:1},
  $\ip{ D_h z, y}_h = \ip{ D z, y}_W=0$, so $y \in \lprp{D(V)}$. Hence
  $y$ is in $V^*$ by Lemma~\ref{lem:V1}.
\end{proof}

\begin{lemma}
  \label{lem:infsupinterface}
  Suppose~\eqref{eq:asm:1} holds. Then, 
  for all $q \in Q$, 
  \[
  \inf_{ v \in D_{h,V}^{-1}(\{ q \})  } \| v \|_W
  \; = \sup_{0 \ne y \in W_h^*}  
  \frac{ | \ip{ q, y}_h| }
       { \quad\| y\|_{W_h^*}}.
  \]
\end{lemma}
\begin{proof}
  The supremum, denoted by $s$,
  is attained by the function $\ut_q$ in $W_h^*$ satisfying
  \begin{gather}
    \label{eq:supuq}
      (A_h^* \ut_q, A_h^* y)_\om + ( \ut_q, y)_\om = -\ip{ q, y}_h,
      \qquad \forall y \in W_h^*, \text { and }
      \\ 
    \label{eq:supnorm}
    s = \| \ut_q \|_{W_h^*}.  
  \end{gather}
  Choosing $y \in \D(\om)^l$ in~\eqref{eq:supuq}, we conclude that the
  distribution $A (A_h^* \ut_q)$ is in
  $L^2(\om)^l$. Hence~\eqref{eq:Dh} is applicable with
  $w = A_h^* \ut_q$ and we obtain
  \begin{subequations}
    \label{eq:AAT}
    \begin{align}
      \label{eq:AAT0}
      A_h A_h^* \ut_q + \ut_q & = 0 
      \\
      \label{eq:AAT1}
      D_h A_h^* \ut_q & = q.
    \end{align}
  \end{subequations}
  Now let $u_q = A_h^* \ut_q$. Then~\eqref{eq:AAT0} implies
  $ A_h u_q =-\ut_q $, which implies
  $A_h^* A_h u_q =- A_h^* \ut_q= -u_q$.  Combining with
  \eqref{eq:AAT1}, we have
  \begin{subequations}
    \label{eq:ATA}
    \begin{align}
      \label{eq:ATA0}
      A_h^* A_h u_q + u_q  & =0
      \\
      D_h  u_q & = q. 
    \end{align}
  \end{subequations}
  Next, we show that $u_q$ is in $V$. By~\eqref{eq:asm:1}, it suffices
  to prove that $u_q \in \lprp{D^*(V^*)}$.  For any $\vt$ in~$V^*,$
  we have,  using  Lemma~\ref{lem:1}, 
  $ \ip{D^* \vt, u_q}_W 
  = 
  -\overline{\ip{D u_q, \vt}}_{W^*} 
  =
  -\overline{\ip{ D_h
      u_q, \vt}}_h =-\overline{\ip{ q, \vt}}_h$.
  The last term is zero because
  $q = D_h z$ for some $z \in V$ and
  $\ip{ q, \vt} = \ip{D z, \vt}_{W^*} = 0$ by
  Lemma~\ref{lem:V1}. Hence $u_q \in \lprp{D^*(V^*)} = V.$

  The infimum of the lemma is $\| q \|_Q.$ By virtue
  of~\eqref{eq:supnorm}, to complete the proof, it suffices to show
  that $\| q \|_Q = \| u_q \|_W = \| \ut_q \|_{W_h^*}.$ The last
  equality is obvious from $u_q= A_h^* \ut_q$ and $ A_h u_q =-\ut_q$,
  hence we need only show that $\| q \|_Q = \| u_q \|_W$. Standard
  variational arguments show that the infimum defining $\| q \|_Q$ is
  attained by a unique minimizer $v_q \in V$ satisfying $\| q \|_Q = \| v_q \|_W$,  
  $D_h v_q = q$
  and $ (A_h v_q, A_h v)_\om + (v_q, v)_\om = 0 $ for all
  $v \in D_{h,V}^{-1}(\{0\}).$
\rev{Choosing 
a $v$ in  $\D(K)$ (whose extension by zero is in 
 $D_{h,V}^{-1}(\{0\})$), we conclude that 
distribution $A^* (A_h v_q)|_K$ is in $L^2(K)^m$ for any $K \in \oh$. Therefore $A_h^* A_h v_q$ is in $L^2(\om)^m$.  In view
  of~\eqref{eq:ATA}, this means that $v_q = u_q$. }
\end{proof}

\begin{proof}[Second proof of Theorem~\ref{thm:wellposed-mesh}]
  According to \cite[Theorem~3.3]{carstensen1}, it suffices to prove
  that there are positive constants $c_0, \hat c$ such that
  \begin{align}
    \label{eq:infsupvolume}
    c_0 \| u \|_\om \le \sup_{0 \ne y \in Y_0}
    \frac{ | (u, A_h^* y)_\om |}{\| y\|_{W_h^*}}
    && \forall \;&  u \in L^2(\om)^m,
    \\
    \label{eq:infsupinterface}
  \hat c \,\| q \|_{Q} 
  \le \sup_{0 \ne y \in W_h^*} 
  \frac{ | \ip{ q, y}_h| }
       { \quad\| y\|_{W_h^*} }
    && \forall \;& q \in Q,
  \end{align}
  where
  $ Y_0 = \{ y \in W_h^* : \; \ip{q, y}_h=0 \text{ for all } q \in
  Q\}.$ 

  Since~\eqref{eq:infsupinterface} follows with $\hat c =1$ from
  Lemma~\ref{lem:infsupinterface}, we only need to
  prove~\eqref{eq:infsupvolume}. First note that since $V$ is closed
  (by~\eqref{eq:asm:1}), $A$ is a closed operator.
  By~\eqref{eq:asm:2}, the range of $A$ is closed.  By the Closed
  Range Theorem for closed operators, range of $A^*$ is closed.  Also,
  the well-known identity $\ker(A^*) = \rprp{\ran(A)}$, in combination
  with~\eqref{eq:asm:2}, implies that $A^*$ is injective. Hence there
  exists a $C>0$ such that
  \begin{equation}
    \label{eq:A'bddbelow}
    C \| y \|_\om \le \| A^* y\|_\om
    \qquad \forall  y \in \dom(A^*)  = V^*.
  \end{equation}
  This implies the following inf-sup condition:
    \[
    C \| y \|_{W^*} \le \sup_{u \in L^2(\om)^m} 
    \frac{ |(u, A^*y)_\om| }{ \| u \|_\om}
    \qquad \forall 
    y \in V^*.
    \] 
    To complete the proof, we note that by standard arguments the
    order of arguments in the inf and sup  may be reversed to get
    \[
      \inf_{u \in L^2(\om)^m} \sup_{y \in V^*} 
      \frac{ | (u, A^*y)_\om| }{\; \; \| u \|_\om \;\| y \|_{W^*}}
      \;=\;
      \inf_{y \in V^*} \sup_{u \in L^2(\om)^m}
      \frac{ | (u, A^*y)_\om| }{  \;\;\| u \|_{\om}\; \| y \|_{W^*}} 
      \ge C.
    \]
    By Lemma~\ref{lem:unbroken}, $Y_0 = V^*$, thus completing the
    proof of~\eqref{eq:infsupvolume}.
\end{proof}

\begin{remark}
  The idea behind Lemma~\ref{lem:infsupinterface} (to consider the two
  related problems~\eqref{eq:AAT} and \eqref{eq:ATA}, one with
  essential boundary conditions and the other with natural boundary
  conditions) was first presented in~\cite{carstensen1,
    CarstDemkoGopal15}, tailored to the specific needs of a Maxwell
  problem. A generalization for first order operators 
  was presented later in~\cite{Wiene16}.
  The argument to prove~\eqref{eq:infsupvolume} using 
  the Closed Range Theorem, was first presented for the case of
  first order Sobolev spaces
  in~\cite[Theorem~6.6]{carstensen1}.
\end{remark}

\bibliography{SchrodingerRef}
\bibliographystyle{siam}

\end{document}